\documentclass[11pt]{amsart}
\usepackage{graphicx}
\usepackage{latexsym}
\usepackage{amsfonts,amsmath,amssymb}
\newtheorem{theorem}{Theorem}[section]

\newtheorem{lemma}[theorem]{Lemma}

\newtheorem{corollary}[theorem]{Corollary}
\newtheorem{example}[theorem]{Example}

\def\bi{\bigskip\noindent}
\def\si{\smallskip\noindent}
\def\pr{\text{P}}
\def\ex{\text{E}}
\def\la{\lambda}
\def\eps{\varepsilon}

\def\cal{\mathcal}

\date{\today}

\begin{document}

\title[Product of cycles]{On the cycle structure of the product of random maximal cycles}

\author{Mikl\'os B\'ona}
\address{Department of Mathematics, University of Florida, $358$ Little Hall, PO Box $118105$,
Gainesville, FL, $32611-8105$ (USA)}
\email{bona@ufl.edu}
\author{Boris Pittel}
\address{Department of Mathematics, The Ohio State University, $231$ West $18$-th Avenue, Columbus, Ohio $43210-1175$ (USA)}
\email{bgp@math.ohio-state.edu}



\begin{abstract} The subject of this paper is the cycle structure of the random permutation $\sigma$ of $[N]$, which is the product of $k$ independent random cycles of maximal length $N$.  We use the character-based Fourier transform to study the counts of cycles of $\sigma$ by length and also the distribution of the elements of the subset $[\ell]$ among the cycles of $\sigma$. 
 \end{abstract}

\maketitle
\section{Introduction}
Enumeration of  permutations of a set $[N]=\{1,2,\dots,N\}$ according to the numbers of cycles of various lengths has a long and glorious history. The plentiful results are not infrequently cast in the
probabilistic light, if the assumption is made that a permutation is chosen {\em uniformly at random }
among all $N!$ permutations. The techniques vary widely, from bijective methods to
multivariate generating functions to functional limit theorems, allowing to find solutions, exact or
asymptotic, of rather delicate, enumerative-probabilistic, problems. More recently there has been
a growing interest in the probabilities regarding distribution of the elements of a subset $S\subseteq [N]$ 
among the cycles of the random permutation. For instance, we can determine the probability that each of the entries in $S$ will be in  a different cycle, or that all entries of $S$ will  be in the same cycle, or that each cycle of $p$ 
will contain at least one entry of $S$. See Lov\'asz \cite{Lovasz} for results of this kind. 

The classic, and more recent,  problems become much more difficult if instead of the uniformly random permutation, we consider a random permutation which is a {\em product} of  random {\em maximal} cycles. That is, our sample space is now that of all ordered $k$-tuples $(p_1,p_2,\cdots ,p_k)$, where all $p_i$ are maximal cycles of length $N$. One can investigate the random permutation
$\sigma:=p_1\cdots p_k$ under the assumption that $p_1,\dots,p_k$ are  maximal cycles,
chosen uniformly at random, and independently of each other, from all $(N-1)!$ such cycles.

 \subsection{Motivation and recent results}
Among the sources of our inspiration are Zagier's formula for the distribution of the number of cycles in $\sigma$ for $k=2$, and the more recent results by Stanley  \cite{Stanley2} and Bernardi et al. \cite{Bernardi}, again for $k=2$. For instance, in \cite{Bernardi}  a formula is proved for the probability that 
$\sigma$, the product of two maximal cycles, separates the {\em given disjoint} subsets of $[N]$, i.e. no two 
of those subsets are represented in the same cycle of $\sigma$. In particular,
the probability that $\sigma$ separates the entries $1,\dots,\ell$ is equal to $1/\ell!$ if $N-\ell$ is odd. In other words, in this aspect, the product of two independent maximal cycles behaves as the uniformly random permutation!

 Beside their intrinsic interest, solutions of the mentioned problems  may lead to surprising applications. 
In \cite{bona-flynn}, B\'ona and Flynn used a result of Stanley \cite{Stanley2} concerning the special case $S=\{1,2\}$ and $k=2$
to prove an exact formula for the average number of block interchanges needed to sort a permutation, a problem motivated
by genome sorting. Equally interesting are the methods that can be used, as they come from a wide array of areas in mathematics, such as character theory, multivariate Gaussian integration,  bijective combinatorics and the summation techniques for hypergeometric sums.

\subsection{Overview: methods and results} 

In $1986$ Harer and Zagier \cite{HarerZagier} discovered a remarkable formula for the bivariate generating 
function of the number of cycles in the product of a maximal cycle and the random, fixed-point
free, involution of $[2n]$, thus solving a difficult problem of enumerating the chord diagrams by the
genus of an associated surface. The proof was based on evaluation of the multidimensional 
Gaussian integrals. Soon after Jackson \cite{Jackson} and later Zagier \cite{Zagier} found
alternative proofs that used characters of the symmetric group $S_{2n}$. Recently the second
author \cite{Pittel} found a different, character-based proof. Its core is computing and marginally
inverting the Fourier transform of the underlying probability measure on $S_{2n}$. In the present paper, we use the techniques in \cite{Pittel}, see also an earlier paper by Chmutov
and Pittel \cite{Chmutov}, to investigate the product of $k$ maximal cycles in $S_N$. To make the discussion reasonably self-contained we will introduce the necessary definitions and facts from \cite{Pittel} in Section \ref{Prelim}. 

We begin Section \ref{cycles} with Lemma \ref{cycledist} that  states an explicit formula for the probability distribution of the
number of cycles in $\sigma$, the product of $k$ random, independent, maximal
cycles in $S_N$. Not surprisingly, the distribution is expressed through the Stirling numbers of
first kind.  In particular, this formula yields the known results, Stanley \cite{Stanley}, for the probabilities that $\sigma$ is
the identity permutation, or that  $\sigma$ is a maximal cycle.  Our analysis also delivers a well-known formula found by Zagier  for the case $k=2$. See
Corollary \ref{zagier} for this special case; see the Appendix by Zagier in Lando and Zvonkin \cite{LandoZvonkin} for the original result of Zagier. 
In Corollary \ref{zagier3}, we also obtain a bivariate generating function for the distribution of the number of cycles for the product of three cycles. We conclude this section with a relatively
compact, integral formula for the probability that the product of two cycles belongs to a given
conjugacy class.

Then, in Section \ref{A},  we turn to the following general question.  Let $p_A(N,\ell;k)$ be 
the probability that the number of elements of $[\ell]=\{1,2,\cdots ,\ell\}$ in each cycle of $\sigma $ comes from the
set $A\subseteq \Bbb Z_{\ge 0}$. What can we say about $p_A(N,\ell;k)$?

To this end, for a general $A$, we first enumerate the admissible permutations by the cycle counts and then evaluate
the sum of character values over all admissible permutations for irreducible representations
labeled by one-hook Young diagrams.  Then we consider the special case when $A= \Bbb Z_{>0}$, i.e. when each cycle of $\sigma $ contains at least one
element of $[\ell]$. Using the inverse Fourier transform, we find an alternating sum expression for this probability with 
$N-\ell+1$ binomial-type summands. This result is proved in Theorem \ref{formulaforA1}. For $k=2$,  this sum reduces to two
notably simpler  expressions, that can be  efficiently computed for moderate $\ell$ and moderate $N-\ell$
respectively.

Next we investigate the case of $A=\{0,\ell\}$, that is, when all elements of $\ell$ are in the same cycle of $\sigma$. 
This computation is longer than its counterpart in the previous case, and it leads to a general formula for $p_A(N,\ell;k)$, given in Theorem \ref{formulaforA2}, that is
analogous to that for $A=\Bbb Z_{>0}$. Again, if $k=2$, then the formula shrinks to a pair of computationally efficient sums for moderate $\ell$ and moderate $N-\ell$ respectively.
For $\ell=2$ and $\ell=3$, we recover the results obtained by Stanley \cite{Stanley2}. 

Having experimented with Maple, we feel confident  that the residual  sums for $k=2$ in either of the two cases do not have a more compact presentation.

After this, in Section \ref{separate}, we turn to our most general problem. We consider disjoint subsets ${\cal S}_1,{\cal S}_2,\cdots ,{\cal S_t}$ of $[N]$ so 
that $|{\cal S}_j|=\ell _j$;  define $\ell=\sum_j \ell_j$. Let $p(N,\vec\ell;k)$ denote the probability that no cycle of $\sigma$
contains elements from more than one
${\cal S}_j$, a property to which we refer by saying that $\sigma$ {\em separates} the sets 
${\cal S}_1,{\cal S}_2,\cdots, {\cal S}_t$.
Bernardi et al. \cite{Bernardi} found a striking formula for $p(N,\vec\ell;2)$ that contained an alternating sum of $\ell-t+1$ terms. Remarkably, the  factor $\prod_j\ell_j!$ aside, the rest of the formula depends on
$\ell$ and $t$ only. In Lemma \ref{generalk}, we show that the separation probability continues to have this
latter property  for all $k\ge 2$, and find an alternating sum formula with $N-\ell+t+1$ terms for this probability, which is computationally efficient if  $t$ and $N-\ell$ are both bounded as $N$ grows. Then, for $k=2$, we are able to simplify this formula to one that is close in appearance, but is significantly different from
 the formula in \cite{Bernardi}. This formula is given in Theorem \ref{k=2separate}, and it  still contains a sum of $\ell-t+1$ summands, but the signs are no longer alternating.  

Finally, in Section \ref{isolated}, we consider the following question. Let us say that the elements of $[\ell]$ are blocked in a permutation 
$s$ of $[N]$ if no two elements of $[\ell]$ are neighbors, {\em and} each element of $[\ell]$ has a neighbor from 
$[N] \setminus [\ell]$. Then, for a general $k\ge 2$,  we find a {\em two-term} formula for the probability that $\sigma$ blocks the elements of $[\ell]$.
This formula is proved in Theorem \ref{twotermformula}.

While on occasion our proofs deliver the already known results, 
we hope that the employed techniques can be used
for a broader variety of problems on cyclic structure of the products of random permutations. 

\section{Preliminaries}\label{Prelim}
A key observation is that the set of all maximal cycles forms a {\em conjugacy class} in the symmetric group $S_N$, a class with particularly simple character values. We mention that permutations generated by a given conjugated class were studied
for instance by Diaconis \cite{diaconis1,diaconis2},  Lulov and Pak \cite{lulov}, and, from a more algebraic point of view, 
by Liebeck, Nikolov, and Shalev \cite{liebeck}.

Let us start with the Fourier inversion formula for a general probability measure $P$ on $S_N$:\begin{equation}\label{geninv}
P(s)=\frac{1}{N!}\sum_{\la\vdash N}f^{\la}\,\text{tr}\bigl(\rho^{\la}(s^{-1})\hat P(\rho^{\la})\bigr);\quad s\in S_N;
\end{equation}
see  Diaconis and Shahshahani \cite{DiaconisShahshahani} and Diaconis \cite{Diaconis}.
Here $\la$ is a generic partition of the integer $N$, $\rho^{\la}$ is the irreducible representation of $S_N$
associated with $\la$, $f^{\la}=\text{dim}(\rho^{\la})$, and $\hat P(\rho^{\la})$ is the $f^{\la}\times f^{\la}$ matrix-valued Fourier transform of $P(\cdot)$ evaluated at $\rho^{\la}$, 
$
\hat P(\rho^{\la})=\sum_{s\in S_N} \rho^{\la}(s) P(s).
$
Let us evaluate the right-hand side of \eqref{geninv} for $P=P_{\sigma}$, the probability measure on $S_N$
induced by $\sigma=\prod_{j=1}^k \sigma_j$, where $\sigma_j$ is uniform on a conjugacy 
class $C_j$.  As the $\sigma_j$ are independent, we have that  $P_{\sigma}(s)=\sum_{s_1,\dots, s_k}\prod_jP_{\sigma_j}(s_j)$, ($s_1\cdots s_k=s$), that is, $P_{\sigma}$ is the convolution of $P_{\sigma_1},\dots,P_{\sigma_k}$. So, by multiplicativity
of the Fourier transform for convolutions, 
$
\hat{P}_{\sigma}(\rho^{\la}) = \prod_j \hat {P}_{\sigma_j}(\rho^{\la}).
$
Since each $P_{\sigma_j}$ is supported by the single conjugacy class $C_j$, we have
$\hat P_{\sigma_j}(\rho^{\la})=\tfrac{\chi^{\la}(C_j)}{f^{\la}}\,I_{f^{\la}}$, $I_{f^{\la}}$ being the
$f^{\la}\times f^{\la}$ identity matrix, see \cite{Diaconis}. So
\[
\hat{P}_{\sigma}(\rho^{\la}) = \prod_{j=1}^k\hat {P}_{\sigma_j}(\rho^{\la})=
(f^{\la})^{-k}\prod_{j=1}^k\chi^{\la}(C_j)\,I_{f^{\la}},
\]
and  \eqref{geninv} becomes
\begin{equation}\label{3chi}
\begin{aligned}
P_{\sigma}(s)&=\frac{1}{N!}\sum_{\la}(f^{\la})^{-k+1}\,\left(\prod_{j=1}^k\chi^{\la}(C_j)\right)\,\text{tr}
\bigl(\rho^{\la}(s^{-1})I_{f^{\la}}\bigr)\\
&=\frac{1}{N!}\sum_{\la}(f^{\la})^{-k+1}\chi^{\la}(s)\prod_{j=1}^k\chi^{\la}(C_j);
\end{aligned}
\end{equation}
see Stanley \cite{Stanley}, Exercise 7. 67.

{\bf Note.\/} For the special case $s=\text{id}$, the identity \eqref{3chi} becomes
\[
\pr_{\sigma}(\text{id})=\frac{1}{N!}\sum_{\la}(f^{\la})^{-k+2}\prod_{j=1}^k\chi^{\la}(C_j).
\]
Since the left-hand side is just ${\cal N}(C_1,\dots,C_k)$, the number of ways to write the identity
permutation as the product of elements of $C_1,\dots,C_k$, divided by $\prod_{j=1}^k |C_j|$,
we obtain the well-known $S_N$-version of Frobenius's  identity 
\begin{equation}\label{Frob}
{\cal N}(C_1,\dots,C_k)=\frac{\prod_{j=1}^k|C_j|}{N!}\sum_{\la}(f^{\la})^{-k+2}\prod_{j=1}^k\chi^{\la}(C_j).
\end{equation}

We will use \eqref{3chi} for $C_j\equiv \cal C_N$, where $\cal C_N$ is the conjugacy class of all
maximal cycles.  By the Murnaghan-Nakayama rule, Sagan \cite{Sagan} (Lemma 4.10.2) or
Stanley \cite{Stanley} (Section 7.17, Equation (7.75)), $\chi^{\la}(\cal C_N)=0$  unless the
diagram $\la$ is a single hook $\la^*$, with one row of length $\la_1$ and one column of height $\la^{1}$, so $\la_1+\la^1=N+1$.  In that case 
\begin{equation}\label{chiN}
\chi^{\la}({\cal C}_N)=(-1)^{\la^1-1}. 
\end{equation}
As for $f^{\la^*}$, the number of Standard Young Tableaux of shape $\la^*$,  applying the hook length formula (or simply selecting the entries that go in the first column),
 we obtain
\begin{equation}\label{fla*=}
f^{\la^*}=\frac{N!}{N\prod_{r=1}^{\la_1-1} r\,\prod_{s=1}^{\la^1-1}s}=\binom{N-1}{\la_1-1}.
\end{equation}
The equations \eqref{3chi}, \eqref{chiN} and \eqref{fla*=} imply 
\begin{equation}\label{Psigma(s)}
P_{\sigma}(s)=\frac{1}{N!}\sum_{\la^*}(-1)^{k(\la^1-1)}\binom{N-1}{\la_1-1}^{-k+1}\chi^{\la^*}(s).
\end{equation}
By the Murnaghan-Nakayama rule, given a hook diagram $\la^*$, the value of  $\chi^{\la^*}(s)$ depends on $s$ only through $\vec{\nu}=\vec{\nu}(s):=\{\nu_r\}_{r\ge 1}$, where $\nu_r=\nu_r(s)$ is the total number of $r$-long cycles in the permutation $s$. It was proved in \cite{Pittel} that 
\begin{equation}\label{chila*(nu)}
\chi^{\la^*}(s)=(-1)^{\la^1+\nu}\,[\xi^{\la_1}]\, \frac{\xi}{1-\xi}\prod_{r\ge 1}(1-\xi^r)^{\nu_r},
\end{equation}
$\nu(s):=\sum_r\nu_r(s)$ being the total number of cycles of $s$. From \eqref{chila*(nu)} it follows that 
\begin{equation}\label{sumchila*,nu}
\begin{aligned}
\sum_{s:\, \vec{\nu}(s)=\vec\nu}\chi^{\la^*}(s)&=(-1)^{N}N! \,{\cal A}(N,\nu,\la_1),\\
{\cal A}(N,\nu,\la_1)&:=\binom{N-1}{N-\la_1}\sum_{\ell\ge 1}(-1)^{\ell}\,\frac{s(\ell,\nu)}{\ell!}
\binom{N-\la_1}{N-\ell},
\end{aligned}
\end{equation}
where $s(\ell,\nu)$ is the signless, first-kind,  Stirling number  of 
permutations of $[\ell]=\{1,2,\cdots ,\ell\}$ with $\nu$ cycles; see the proof of Theorem 2.1 and the equation (2.20) in \cite{Pittel}. 
The formulas \eqref{3chi}, \eqref{chila*(nu)} and \eqref{sumchila*,nu} are the basis of the proofs that follow.

\section{Distribution of the number of cycles in $\sigma$}\label{cycles}  To stress dependence 
of $\sigma$ on $k$, in this section we will write $\sigma^{(k)}$ instead of $\sigma$. The following lemma
will be useful in our computations. 

Combining 
\eqref{sumchila*,nu} and \eqref{Psigma(s)}, and using $\la^1+\la_1=N+1$, we obtain the following formula. 

\begin{lemma} \label{cycledist} The identity 
\begin{equation}\label{P(nu(sigma)=nu)}
\begin{aligned}
\pr(\nu(\sigma^{(k)})=\nu)&=(-1)^N\sum_{\la_1=1}^N(-1)^{k(N-\la_1)}\binom{N-1}{N-\la_1}^{-k+2}\\
&\quad\times\sum_{\ell\ge 1}(-1)^{\ell}\,\frac{s(\ell,\nu)}{\ell!}
\binom{N-\la_1}{N-\ell}
\end{aligned}
\end{equation}
holds. 
\end{lemma}

\begin{proof}
Combination of
\eqref{sumchila*,nu}, \eqref{Psigma(s)}, and $\la^1+\la_1=N+1$ proves \eqref{P(nu(sigma)=nu)}. 
\end{proof}

\begin{corollary} \label{sigma-id} For $k\ge 2$, the identiy 
\begin{equation}\label{Pr(id)=}
\begin{aligned}
\pr(\sigma^{(k)}=\text{id})=&\,\pr(\nu(\sigma)=N)
=\frac{1}{N!}\sum_{r=0}^{N-1}(-1)^{kr}\binom{N-1}{r}^{-k+2},
\end{aligned}
\end{equation}
holds.
\end{corollary}

\begin{proof}
Use formula \eqref{P(nu(sigma)=nu)} and the fact that $s(\ell,\nu)=0$ for $\ell<\nu$.
\end{proof}

Note that formula (\ref{Pr(id)=}) appears as equation (7.181) in \cite{Stanley}.  In the special case
of $k=2$  Corollary \ref{sigma-id} yields 
 \begin{equation}\label{P(sigma2isid)}
 \pr(\sigma^{(2)}=\text{id})=\tfrac{1}{(N-1)!}. 
 \end{equation}
 This is an obvious result, since the inverse of the uniformly random cycle is again the uniformly
 random cycle.
 
The special case of $k=3$ is not so obvious. However, combining \eqref{Pr(id)=} and  the identity 
\begin{equation}\label{Sury}
\sum_{r=a}^n \frac{(-1)^r}{\binom{n}{r}}=\frac{n+1}{n+2}\left[\frac{(-1)^a}{\binom{n+1}{a}}+(-1)^n\right]
\end{equation}
(Sury \cite{Sury}, Stanley \cite{Stanley}, equation (7.211),  Sury et al. \cite {SWZ}), we have  a non-obvious answer
\begin{equation}\label{P(siigma3=id)}
\pr(\sigma^{(3)}=\text{id})=\frac{1+(-1)^{N-1}}{(N-1)!(N+1)},
\end{equation}
see \cite{Stanley}, Exercise 7.67 (d).

The remarkable identity \eqref{Sury}  followed from the elementary, yet surprisingly powerful, 
formula
\begin{equation}\label{recbin=}
\binom{n}{r}^{-1}=(n+1)\int_0^1t^r(1-t)^{n-r}\,dt.
\end{equation}

Note that for the even $N$,  equation \eqref{P(siigma3=id)} returns zero probability, and that is how it should be, since the product of three
even cycles is an odd permutation, and therefore, cannot be the identity. Furthermore, since
$\sigma^{(k)}=\sigma^{(k-1)}\sigma_k$, $\sigma^{(k)}$ is the identity iff $\sigma^{(k-1)}=
(\sigma_{k})^{-1}$, which is a maximal cycle. As $(\sigma_{k})^{-1}$ is uniform on the set of all
$(N-1)!$ maximal cycles,  {\it and\/} independent of $\sigma^{(k-1)}$, we see then that
\begin{equation}\label{P(cycle)=P(id)}
\pr(\sigma^{(k-1)}\text{ is a cycle})=(N-1)!\pr(\sigma^{(k)}=\text{id}).
\end{equation}

In the special case of $k=2$, we rediscover a result that has been proved several times, with different methods. 
\begin{corollary} \label{cycleproducts}
We have
\begin{equation}\label{P(sigma2=cycle)}
\pr(\sigma^{(2)}\text{ is a cycle})=\frac{1+(-1)^{N-1}}{N+1}.
\end{equation}
\end{corollary}

\begin{proof} Immediate from 
  equations \eqref{P(siigma3=id)} and \eqref{P(cycle)=P(id)}.
\end{proof}

For even $N$, the statement of Corollary \ref{cycleproducts} is obvious, since the product of two maximal cycles is an even permutation, and hence, it cannot be an $N$-cycle for even $N$.  For odd $N$, the result  is 
equivalent to a well-known, but not at all obvious, fact that there are $\tfrac{2(N-1)!}{N+1}$ ways to
factor a given maximal cycle into a product of two maximal cycles; see for instance
\cite{cangelmi} and the references therein.  In general, the equations \eqref{Pr(id)=}, \eqref{P(cycle)=P(id)} imply 
the following. 

\begin{corollary}
For all positive integers $k$, the formula
\begin{equation}\label{P(cycle)gen}
\pr(\sigma^{(k)}\text{ is a cycle})=\frac{1}{N}\sum_{r=0}^{N-1}(-1)^{(k+1)r}\binom{N-1}{r}^{-k+1}.
\end{equation}
holds.
\end{corollary}

Further, it follows from \eqref{P(nu(sigma)=nu)} that for every real number $x$, we have
\begin{align}
\ex\bigl[x^{\nu(\sigma)}\bigr]&=(-1)^N\sum_{\la_1=1}^N(-1)^{k(N-\la_1)}\binom{N-1}{N-\la_1}^{-k+2}\notag\\
&\quad\times\sum_{\ell\ge 1}\frac{(-1)^{\ell}}{\ell!}
\binom{N-\la_1}{N-\ell}\sum_{\nu\ge 1}x^{\nu}s(\ell,\nu)\notag\\
&=(-1)^N\sum_{\la_1=1}^N(-1)^{k(N-\la_1)}\binom{N-1}{N-\la_1}^{-k+2}\sum_{\ell\ge 1}\binom{N-\la_1}{N-\ell}\binom{-x}{\ell}\notag\\
&=(-1)^N\sum_{\la_1=1}^N(-1)^{k(N-\la_1)}\binom{N-1}{N-\la_1}^{-k+2}\binom{N-\la_1-x}{N}\notag\\
&=(-1)^N\sum_{r=0}^{N-1}(-1)^{kr}\binom{N-1}{r}^{-k+2}\binom{r-x}{N}.\label{k,Exnu(sigma)}
\end{align}
For a positive integer $x$, the non-zero contributions to the sum come from $r<\min\{N,x\}$. So,
for instance,
\begin{align*}
\ex\bigl[2^{\nu(\sigma^{(k)})}\bigr]=&\,N+1+\frac{(-1)^k}{(N-1)^{k-2}},\quad (N>1),\\
\ex\bigl[3^{\nu(\sigma^{(k)})}\bigr]=&\,2(N+2)_2-\frac{N+1}{(N-1)^{k-2}}+\binom{N-1}{2}^{-k+2},
\quad (N>2),
\end{align*}
where we use the notation $(a)_b=a(a-1)\cdots (a-b+1)$, for integers $a\geq b\geq 0$.

For $k=2$ and $x>N$,  equation \eqref{k,Exnu(sigma)} implies the following formula. 

\begin{corollary} \label{zagier}
The identity
\begin{align}
\ex\bigl[x^{\nu(\sigma^{(2)})}\bigr]&=(-1)^N\sum_{\la_1=1}^N\binom{N-\la_1-x}{N}\notag\\
&=\sum_{\la_1=1}^N\binom{\la_1+x-1}{N}
=\sum_{j=N}^{N+x-1}\binom{j}{N}-\sum_{j=N}^{x-1}\binom{j}{N}\notag\\
&=\binom{N+x}{N+1}-\binom{x}{N+1}=\binom{N+x}{N+1}+(-1)^N\binom{N-x}{N+1}
\label{2,Exnu(sigma)} 
\end{align} holds. 
\end{corollary}
Of course, the identity \eqref{2,Exnu(sigma)} holds for all $x$. It is equivalent to Zagier's result, (see  the Appendix by Zagier in Lando and Zvonkin \cite{LandoZvonkin}), stating 
that 
\[
\pr(\nu(\sigma^{(2)})=\nu)=(1+(-1)^{N-\nu})\,[x^{\nu}]\binom{N+x}{N+1}.
\]

For $k=3$, we can prove the following analogue of Corollary \ref{zagier}. 


\begin{corollary} \label{zagier3}
 \begin{equation}\label{k=3,H-Z}
\sum_{N\ge 1}\frac{y^N}{N}\ex\bigl[x^{\nu(\sigma^{(3)}(N))}\bigr]=
\int_0^1\frac{(1-y(1-t))^{-x} -(1-y(1-t))^x}{1-yt(1-t)}\,dt;
\end{equation}
where $\sigma^{(3)}(N)$ is the product of $3$ random cycles of length $N$, and $|x|\le1$, $|y|<1$.
\end{corollary}

Note that the right-hand side of \eqref{k=3,H-Z} is an odd function of $x$, which should be expected, 
since --regardless of the parity of 
$N$-- the number of cycles in $\sigma^{(3)}(N)$ is odd.  In particular, differentiating both
sides at $x=1$, we obtain that  for $y\in [0,1)$,

\begin{multline*}
\sum_{N\ge 1}\frac{y^N}{N}\pr(\sigma^{(3)}(N)\text{ is a cycle})=2\int_0^1\frac{\log\bigl(1-y(1-t)\bigr)^{-1}}{1-yt(1-t)}\,dt\\
=2\sum_{j>0}\frac{1}{j}\int_0^1\frac{\bigl(y(1-t)\bigr)^j}{1-yt(1-t)}\,dt=2\sum_{j>0,h\ge 0}
\frac{y^{j+h}}{j}\int_0^1(1-t)^{j+h}t^h\,dt\\
=2\sum_{j>0,h\ge 0}\frac{y^{j+h}}{j}(j+2h+1)^{-1}\binom{j+2h}{h}^{-1}.
\end{multline*}
So
\begin{equation}\label{p(sigma^3=cycle}
\pr(\sigma^{(3)}(N)\text{ is a cycle})=2N\sum_{h<N}(N-h)^{-1}(N+h+1)^{-1}\binom{N+h}{h}^{-1};
\end{equation}
compare with the equation \eqref{P(cycle)gen} for $k=3$.

\begin{proof} (of Corollary \ref{zagier3})  Since both sides of  \eqref{k=3,H-Z} are analytic for $|y|<1$, it
suffices to prove the identity for $|y|\le 1/3$. From \eqref{k,Exnu(sigma)}, \eqref{recbin=} and
\[
(-1)^N \binom{r-x}{N}= [z^N] (1-z)^{r-x},
\]
we obtain 
\begin{multline*}
N^{-1}\ex\bigl[x^{\nu(\sigma^{(3)}(N))}\bigr]=[z^N] \sum_{r=0}^{N-1}(1-z)^{r-x}\int_0^1(1-t)^{N-1-r}t^r\,dt\\
=[z^N] (1-z)^{-x}\int_0^1(1-t)^{N-1}\sum_{r=0}^{N-1}\left(-\frac{(1-z)t}{1-t}\right)^r\,dt\\
=[z^N] (1-z)^{-x}\int_0^1\frac{(1-t)^N+(-1)^{N+1}\bigl((1-z)t\bigr)^N}{1-tz}\,dt.
\end{multline*}
Next
\[
y^N(1-t)^N [z^N] \,\frac{(1-z)^{x}}{1-tz}=[z^N]\,\frac{\bigl(1-(1-t)yz\bigr)^{-x}}{1-t(1-t)yz};
\]
so
\begin{align}
\int_0^1\sum_{N\ge 1}y^N(1-t)^N [z^N] \,\frac{(1-z)^{-x}}{1-tz}\,dt=&\,\int_0^1\sum_{N\ge 1}[z^N]\,\frac{\bigl(1-(1-t)yz\bigr)^{-x}}{1-t(1-t)yz}\,dt\notag\\
=&\,\int_0^1\frac{\bigl(1-(1-t)y\bigr)^{-x}}{1-t(1-t)y}\,dt  -1.\label{first}
\end{align}
Further, by the Cauchy integral  formula,
\begin{multline*}
y^N[z^N]\frac{(1-z)^{-x}}{1-tz}\bigl((1-z)t\bigr)^N=\frac{1}{2\pi i}\oint_{|z|=2/3}
\frac{(1-z)^{-x}}{z^{N+1}(1-tz)}\bigl(y(1-z)t\bigr)^N\,dz\\
=\frac{1}{2\pi i}\oint_{|z|=2/3}
\frac{(1-z)^{-x}}{z(1-tz)}\left(\frac{y(1-z)t}{z}\right)^N\,dz .
\end{multline*}
On the circle $|z|=2/3$, we have $\bigl|\tfrac{y(1-z)t}{z}\bigr|\le 5|y|/2\le 5/6$; so summing
over $N\ge 1$, 
\begin{align*}
&\sum_{N\ge 1} (-1)^{N+1}y^N[z^N]\frac{(1-z)^{-x}}{1-tz}\bigl((1-z)t\bigr)^N\\
=&\frac{1}{2\pi i}\oint_{|z|=2/3}\frac{(1-z)^{-x}}{z(1-tz)}\frac{\tfrac{y(1-z)t}{z}}{1+\tfrac{y(1-z)t}{z}}\,dz\\
=&\frac{1}{2\pi i}\oint_{|z|=2/3}\frac{(1-z)^{-x}}{z(1-tz)}\cdot\frac{y(1-z)t}{z+y(1-z)t}\,dz.
\end{align*}
For $t>0$, in the circle $|z|\leq 2/3$ the integrand has two poles, both simple, at $z=0$ and 
$z=-\tfrac{yt}{1-yt}$, with respective residues equal $1$ and $-\tfrac{(1-yt)^x}{1-yt(1-t)}$.
Thus
\[
\sum_{N\ge 1} (-1)^{N+1}y^N[z^N]\frac{(1-z)^{-x}}{1-tz}\bigl((1-z)t\bigr)^N=1-\frac{(1-yt)^x}
{1-yt(1-t)}.
\]
Integrating for $t\in [0,1]$ and adding to \eqref{first}, we obtain
\[
\sum_{N\ge 1}\frac{y^N}{N}\ex\bigl[x^{\nu(\sigma^{(3)}(N))}\bigr]=\int_0^1\frac{(1-y(1-t))^{-x} -(1-yt)^x}{1-yt(1-t)}\,dt,
\]
which is equivalent to \eqref{k=3,H-Z}, as $t(1-t)$ is symmetric with respect to $t=1/2$.
\end{proof}

Our final result in this section is a relatively compact, integral, formula for $P_n(\boldsymbol\nu)$,
the probability that $\sigma^{(2)}$ has $\nu_{\ell}$ cycles of length $\ell$, $1\le \ell\le n$, for
the {\it arbitrary\/} $\boldsymbol\nu$, i.e. satisfying the only constraint $\sum_{\ell} \ell\nu_{\ell}=N$.
Since $\sigma^{(2)}$ is even, $P_n(\boldsymbol\nu)=0$ if $\sum_{\ell\text{ even}}\nu_{\ell}$ is odd.

\begin{theorem}\label{gensig2}
\[
P_n(\boldsymbol\nu)=\frac{N}{\prod_{\ell} \ell^{\nu_{\ell}}\nu_{\ell}!}\int_0^1\prod_{\ell\ge 1}
\bigl[t^{\ell}+(-1)^{\ell+1}(1-t)^{\ell}\bigr]^{\nu_{\ell}}\,dt.
\]
\end{theorem}
\begin{proof} First of all the number of permutations $s$ with cycle parameter $\boldsymbol\nu$ is
$N!/\prod_{\ell} \ell^{\nu_{\ell}}\nu_{\ell}!$. Furthermore, for every such permutation $s$, by
 \eqref{chila*(nu)},  $\la^1+\la_1=N+1$ and $\nu=\sum_{\ell}\nu_{\ell}$, we obtain: setting $r=N-\la_1$, and choosing a positive $\rho$,
\[
\chi^{\la^*}(s)=(-1)^{N+r}\frac{1}{2\pi i}\oint\limits_{|\xi|=\rho}\frac{1}{\xi^{r+1}(1-\xi)}\prod_{\ell\ge 1}
(\xi^{\ell}-1)^{\nu_{\ell}}\,d\xi.
\]
here the circular contour is traversed counter-clockwise, and $\rho$ is arbitrary, as the integrand is
singular at $\xi=0$ only. Substituting $\xi = 1/\eta$, we have
\[
\chi^{\la^*}(s)=(-1)^{N+r}\frac{1}{2\pi i}\oint\limits_{|\eta|=1/\rho}\frac{1}{\eta^{N-r}(\eta-1)}
\prod_{\ell\ge 1}(1-\eta^{\ell})^{\nu_{\ell}}\,d\eta,
\]
with the contour traversed counter-clockwise again. Plugging this formula into the equation
\eqref{Psigma(s)}, and using \eqref{recbin=}, we have 
\begin{multline*}
P_{\sigma^{(2)}}(s)=\frac{(-1)^N}{N!}\sum_{r=0}^{N-1}\binom{N-1}{r}^{-1}\chi^{\la^*(s)}(s)\\
=\frac{(-1)^NN}{N!}\sum_{r=0}^{N-1}\chi^{\la^*}(s)\int\limits_0^1t^r(1-t)^{N-1-r}\,dt\\
=\frac{(-1)^N}{(N-1)!}\,\frac{1}{2\pi i}\oint\limits_{|\eta|=1/\rho}\!\!\left(\int\limits_0^1(1-t)^{N-1}\sum_{r=0}^{N-1}
\left(\!-\frac{t\eta}{1-t}\right)^r\,dt\!\right)\frac{\prod_{\ell}(1-\eta^{\ell})^{\nu_{\ell}}}{\eta^N(\eta-1)}\,d\eta\\
=\frac{(-1)^N}{(N-1)!}\int\limits_0^1\left(\frac{1}{2\pi i}\oint\limits_{|\eta|=1/\rho}
\frac{(1-t)^N-(-t\eta)^N}{1-t+t\eta}\cdot\frac{\prod_{\ell}(1-\eta^{\ell})^{\nu_{\ell}}}{\eta^N(\eta-1)}\,d\eta\right)\,dt
\end{multline*}
Pick $\eps\in (0,1)$ and consider $t\le 1-\eps$. Choose $\rho>(1-\eps)/\eps$. For this $\rho$,
the internal integrand has two  singular points, $\eta=0$ and $\eta=-(1-t)/t$,  respectively within and without the integration contour. Crucially, 
\[
\frac{-(-t\eta)^N}{1-t+t\eta}\cdot\frac{\prod_{\ell}(1-\eta^{\ell})^{\nu_{\ell}}}{\eta^N(\eta-1)}
=\frac{-(-t)^N}{1-t+t\eta}\cdot\frac{\prod_{\ell}(1-\eta^{\ell})^{\nu_{\ell}}}{\eta-1}
\]
has no singularity at $\eta=0$,  and for $t>0$
\[
\frac{(1-t)^N}{1-t+t\eta}\cdot\frac{\prod_{\ell}(1-\eta^{\ell})^{\nu_{\ell}}}{\eta^N(\eta-1)}=O(|\eta|^{-2}),
\quad |\eta|\to\infty,
\]
as $\sum_{\ell}\ell\nu_{\ell}=N$. So, by the residue theorem,
the internal integral equals 
\begin{multline*}
\frac{1}{2\pi i}\oint\limits_{|\eta|=1/\rho}
\frac{(1-t)^N}{1-t+t\eta}\cdot\frac{\prod_{\ell}(1-\eta^{\ell})^{\nu_{\ell}}}{\eta^N(\eta-1)}\,d\eta\\
=-t^{-1}\left.(1-t)^N\cdot\frac{\prod_{\ell}(1-\eta^{\ell})^{\nu_{\ell}}}{\eta^N(\eta-1)}\right|_{\eta=-\tfrac{1-t}{t}}\\
=(-1)^N\prod_{\ell\ge 1}\bigl[t^{\ell}+(-1)^{\ell+1}(1-t)^{\ell}\bigr]^{\nu_{\ell}},
\end{multline*}
for all $0< t\le 1-\eps$. Letting $\eps\to 0$, we obtain
\[
P_{\sigma^{(2)}}(s)=\frac{1}{(N-1)!}\int\limits_0^1\prod_{\ell\ge 1}\bigl[t^{\ell}+(-1)^{\ell+1}(1-t)^{\ell}\bigr]^{\nu_{\ell}}.
\]
Multiplying the result by $N!/\prod_{\ell}\ell^{\nu_{\ell}}\nu_{\ell}!$ we complete the proof.
\end{proof}
\begin{corollary}\label{P(N,r)=}  Let $P_{N,r}$ denote the probability that all cycles of $\sigma^{(2)}$ are of the same length $r\ge 2$, i.e. $N\equiv 0\,(\text{mod }r)$ and $\nu_r=N/r$. (So $P_{N,r}=0$ if $r$ is
even and $N\not\equiv 0\,(\text{mod }2r)$.) Then
\begin{equation}\label{P(N,r)}
\begin{aligned}
P_{N,r}&=\frac{N}{r^{N/r}(N/r)!}\int_0^1\bigl[t^r+(-1)^{r+1}(1-t)^r\bigr]^{N/r}\,dt\\
&=\frac{N}{(N+1)r^{N/r}(N/r)!}\sum_{0\le j\le N\atop j\equiv 0(\text{mod }r)}
\!\!\!\!(-1)^{j(r+1)/r}\frac{\binom{N/r}{j/r}}{\binom{N}{j}}\\
\end{aligned}
\end{equation}
In particular,
\begin{equation}\label{P(N,2),P(N,3)}
P_{N,2}=\frac{N}{2^{N/2}(N/2+1)!},\quad P_{N,3}=\frac{N}{(N/3)!\,(12)^{N/3}}\sum_{j=0}^{N/3}
\binom{N/3}{j}\frac{3^j}{2j+1}.
\end{equation}
\end{corollary}
\noindent Derivation of \eqref{P(N,2),P(N,3)} seems to indicate that the second line identity
in \eqref{P(N,r)} is the preferred expression for the probability $P_{N,r}$ when $r>3$.

\begin{proof} 
The second identity in \eqref{P(N,r)} follows from Theorem \ref{gensig2} via binomial
formula for the integrand $\bigl[t^r+(-1)^{r+1}(1-t)^{r}\bigr]^{N/r}$ and term-by-term integration.
The formulas \eqref{P(N,2),P(N,3)} follow immediately by integration from the first identity in
\eqref{P(N,r)},  as
\begin{align*}
&t^2-(1-t)^2 =2t-1,
&t^3+(1-t)^3= 1/4 +3u^2,\quad u=t-1/2.
\end{align*}
\end{proof}
\begin{corollary}\label{inv} For all positive integers $N$, we have
\[
\pr(\sigma^{(2)}\text{ is an involution})=N\!\!\!\sum_{\nu_1+2\nu_2=N\atop \nu_2\text{ even}}\frac{1}{\nu_1!\,2^{\nu_2}(\nu_2+1)!}.
\]
\end{corollary}
The identities equivalent to \eqref{P(N,2),P(N,3)}  were proved in Doignon and Labarre \cite{doignon} by using the sum-type formulas for the total number of ways to represent
a maximal cycle as a product of a maximal cycle and a permutation from a given conjugacy
class, see Goupil \cite{Gou}, Stanley \cite{Sta1}, Goupil and Schaeffer \cite{GouSch}.  The sequence $(N-1)! P_{N,2}$ is listed by
Sloane as A035319, and known as the counts of certain rooted maps, see Walsh and Lehman \cite{walsh}. The sequence 
$(N-1)!P_{N,3}$ is listed in Sloane as A178217.

\section{Probability that  the occupancy numbers  of the cycles of $\sigma$ by the
elements of $[\ell]$ belong to a given set}\label{A}

In the section title and elsewhere below $\sigma$ is $\sigma^{(k)}$, the product of $k$
random maximal cycles.  Let $A\subseteq \Bbb Z_{\ge 0}$ be given. Introduce $p_A(N,\ell;k)$, the probability that the number
of elements of $[\ell]$ in each cycle of $\sigma$ belongs
to the set $A$. 

The examples include: (1) $A_1=\Bbb Z_{>0}$; each cycle must contain at least
one element of $[\ell]$; (2) $A_2=\{0,\ell\}$; one of the cycles of $\sigma$ contains the whole
set $[\ell]$; (3) $A_3=\{0,1\}$; each element of $[\ell]$ belongs to a distinct cycle of $\sigma$.
The case of $k=2$, $\ell=2$ and $A=\{0,2\}$ or $A=\{0,1\}$ was solved by Stanley \cite{Stanley2}.
Very recently Bernardi et al. \cite{Bernardi} solved the case $k=2$, $A=\{0,1\}$ for $\ell \ge 2$.
In fact they solved a general problem of separation probability for $t$ disjoint sets ${\cal S}_1,\dots, {\cal S}_{t}$.

To evaluate  $p_A(N,\ell;k)$, consider first $Q_A(\vec{\nu},\ell)$, the total number of permutations $s$ of $[N]$, with 
$\vec{\nu}(s)=\{\nu_r(s)\}=\{\nu_r\}=\vec{\nu}$, such that the number of
elements of $[\ell]$ in every cycle is an element of $A$. The reason we need $Q_A(\vec{\nu},\ell)$ is that the key formula \eqref{chila*(nu)}
expresses $\chi^{\la*}(s)$ through the cycle counts $\nu_r(s)$, $r\ge 1$.

To evaluate $Q_A(\vec{\nu},\ell)$, introduce the non-negative integers $a_{r,j}$, $b_{r,j}$ that
stand for the generic numbers of elements from $[\ell]$ and $[N]\setminus [\ell]$ in the $j$-th cycle of length $r$, $(j\le \nu_r)$. 

\begin{theorem}
For all $\ell\geq 2$, the identity \begin{equation} Q_A(\vec{\nu},\ell) =
(N-\ell)!\,\ell!\,\, [w^{\ell}]\prod_r\frac{1}{\nu_r!}\left(\frac{\sum_{a\in A}\binom{r}{a}w^a}{r}\right)^{\nu_r}
\end{equation} holds. 
\end{theorem}

\begin{proof} For $\bold a$, $\bold b$ to be admissible we must have
\begin{align}
\!\!a_{r,j}&+b_{r,j}=r,\label{a+b=r}\\
 &a_{r,j}\in A, \label{ainA}\\
\sum_{r,j}a_{r,j}&=\ell,\quad \sum_{r,j}b_{r,j}=N-\ell.\label{sumarj=ell,sumbrj=N-ell}
\end{align}
Therefore
\begin{equation}\label{p(nu,ell)}
\begin{aligned}
Q_A(\vec{\nu},\ell)&=(N-\ell)!\,\ell!\sum_{\bold a,\,\bold b\text{ meet }\atop
\eqref{a+b=r}, \eqref{ainA},\eqref{sumarj=ell,sumbrj=N-ell}}\prod_r\frac{((r-1)!)^{\nu_r}}{\nu_r!}\prod_{j\le r}\frac{1}{a_{r,j}!\,b_{r,j}!}\\
&=(N-\ell)!\,\ell!\,[w^{\ell}]\prod_r\frac{1}{r^{\nu_r}\nu_r!}\,\prod_{j\le\nu_r} \sum_{a_{r,j}\in A}
\binom{r}{a_{r,j}} w^{a_{r,j}}\\
&=(N-\ell)!\,\ell!\,\, [w^{\ell}]\prod_r\frac{1}{\nu_r!}\left(\frac{\sum_{a\in A}\binom{r}{a}w^a}{r}\right)^{\nu_r}.
\end{aligned}
\end{equation}
\end{proof}

So, using \eqref{chila*(nu)} and $\nu=\sum_r\nu_r$, we conclude that
\begin{multline*}
\sum_{s:\,\vec{\nu}(s)=\vec{\nu}}\chi^{\la^*}(s)=(-1)^{\la^1}(N-\ell)!\,\ell!\\
\times 
[\xi^{\la_1}w^{\ell}]\,\,\frac{\xi}{1-\xi}\,\prod_r\frac{1}{\nu_r!}\left(-\frac{(1-\xi^r)\bigl(\sum_{a\in A}\binom{r}{a}w^a\bigr)}{r}\right)^{\nu_r}.
\end{multline*}
Call a permutation $s$ of $[N]$ admissible if the numbers of elements from $[\ell]$ 
 in each cycle of $s$ meet the constraint \eqref{ainA}. The above identity implies
 \begin{multline}\label{sumcharadmsss1}
\sum_{s\text{ admissible}}\chi^{\la^*}(s)=(-1)^{\la^1}(N-\ell)!\,\ell!\\
\times 
[\xi^{\la_1}w^{\ell}]\,\,\frac{\xi}{1-\xi}\,\sum_{\vec\nu:\atop 1\nu_1+2\nu_2+\cdots =N}\prod_r\frac{1}{\nu_r!}\left(-\frac{(1-\xi^r)\bigl(\sum_{a\in A}\binom{r}{a}w^a\bigr)}{r}\right)^{\nu_r}.
\end{multline}
The expression in the second line of  \eqref{sumcharadmsss1} equals
\begin{align}
&[\xi^{\la_1}w^{\ell}x^N] \frac{\xi}{1-\xi}\,\sum_{\vec\nu\,\ge \bold 0}\prod_r\frac{(x^r)^{\nu_r}}{\nu_r!}\left(-\frac{(1-\xi^r)\bigl(\sum_{a\in A}\binom{r}{a}w^a\bigr)}{r}\right)^{\nu_r}\notag\\
=&[\xi^{\la_1}w^{\ell}x^N] \frac{\xi}{1-\xi}\,\prod_r\sum_{\nu_r\ge 0}\frac{1}{\nu_r!}
\left(-\frac{x^r(1-\xi^r)\bigl(\sum_{a\in A}\binom{r}{a}w^a\bigr)}{r}\right)^{\nu_r}\notag\\
=&[\xi^{\la_1}w^{\ell}x^N] \frac{\xi}{1-\xi}\,\prod_r\exp\left(-\frac{x^r(1-\xi^r)\bigl(\sum_{a\in A}\binom{r}{a}w^a\bigr)}{r}\right)\notag\\
=&[\xi^{\la_1}w^{\ell}x^N] \frac{\xi}{1-\xi}\,\exp\left(-\sum_{r\ge 1}\frac{x^r(1-\xi^r)\bigl(\sum_{a\in A}\binom{r}{a}w^a\bigr)}{r}\right).\label{xila1wellxN}
\end{align}
\subsection{Probability that each cycle of $\sigma$ contains at least one element of $[\ell]$}\label{A=A_1} In this case $A=A_1=\Bbb Z_{>0}$. Therefore 
\[
\sum_{a\in A}\binom{r}{a}w^a=(1+w)^r-1.
\]

In this section, we prove the following result and discuss some of its special cases. 
\begin{theorem}\label{formulaforA1}
For all positive integers $\ell$ and $k$, the equality \begin{equation} \label{equationforA1} 
p_{A_1}(N,\ell;k) = \binom{N}{\ell}^{-1}\sum_{\la_1=\ell}^N(-1)^{(k-1)(N-\la_1)}
\binom{N-1}{N-\la_1}^{-k+1}\binom{\la_1-1}{\ell-1} \end{equation}
holds.
\end{theorem}

\begin{proof} Using \eqref{sumcharadmsss1}, \eqref{xila1wellxN} and $\sum_{j\ge 1} z^j/j=-\log(1-z)$,
$|z|<1$, we obtain 
\begin{equation}\label{sumadmisss}
\begin{aligned}
\sum_{s\text{ admissible}}\chi^{\la^*}(s)=&(-1)^{\la^1}(N-\ell)!\,\ell!\\
&\times [\xi^{\la_1}w^{\ell}x^N]\,\,\frac{\xi}{1-\xi}\,\frac{\bigl(1-x(1+w)\bigr)(1-\xi x)}{\bigl(1-\xi x(1+w)\bigr)(1-x)}.
\end{aligned}
\end{equation}
Let us simplify this formula. Write
\begin{align*}
[w^{\ell}]\frac{1-x(1+w)}{1-\xi x(1+w)}=&\,\frac{\xi-1}{\xi}[w^{\ell}] \frac{1}{1-\xi x(1+w)}\\
=&\frac{1-\xi}{\xi^2x}\,[w^{\ell}]\left(w-\frac{1-\xi x}{\xi x}\right)^{-1}\\
=&\frac{1-\xi}{\xi^2 x}\binom{-1}{\ell}(-1)^{-1-\ell}\left(\frac{1-\xi x}{\xi x}\right)^{-1-\ell}\\
=&-\frac{1-\xi}{\xi^2 x}\left(\frac{\xi x}{1-\xi x}\right)^{1+\ell}.
\end{align*}
Therefore 
\begin{multline*}
[\xi^{\la_1}w^{\ell}x^N]\,\,\frac{\xi}{1-\xi}\,\frac{\bigl(1-x(1+w)\bigr)(1-\xi x)}{\bigl(1-\xi x(1+w)\bigr)(1-x)}
=  -[\xi^{\la_1}x^N]\,(1-x)^{-1}\left(\frac{\xi x}{1-\xi x}\right)^{\ell}\\
 =  -[x^N]\, x^{\la_1}(1-x)^{-1}\cdot [y^{\la_1}]\left(\frac{y}{1-y}\right)^{\ell} \\
=  -[y^{\la_1-k}](1-y)^{-\ell}= -\binom{\la_1-1}{\la_1-\ell},
\end{multline*}
where  $\binom{a}{b}=0$ for $b<0$. So \eqref{sumadmisss} becomes
\begin{equation}\label{sumaddexpl}
\sum_{s\text{ admissible}}\chi^{\la^*}(s)=(-1)^{\la^1-1}(N-\ell)!\,\ell!\binom{\la_1-1}{\la_1-\ell}.
\end{equation}
Combining \eqref{sumaddexpl} and \eqref{Psigma(s)} we conclude that 
\begin{equation}\label{Pr(keverywhere)}
\begin{aligned}
p_{A_1}(N,\ell;k)
=&\frac{1}{N!}\sum_{\la^*}(-1)^{k(\la^1-1)}\binom{N-1}{\la_1-1}^{-k+1}\sum_{s\text{ admissible}}\chi^{\la^*}(s)\\
=&\binom{N}{\ell}^{-1}\sum_{\la_1=\ell}^N(-1)^{(k-1)(N-\la_1)}\binom{N-1}{N-\la_1}^{-k+1}\binom{\la_1-1}{\ell-1},
\end{aligned}
\end{equation}
\si
which was to be proved. 
\end{proof}

Note that as $N\to\infty$,  the dominant contribution to the right-hand side in \eqref{Pr(keverywhere)} comes from 
$\la_1=\ell$ and $\la_1=N$, so that $p_{A_1}(N,\ell;k)=\ell/N +O(N^{-2\ell+1})$; the formula
is useful for $\ell>1$. We remark that $\ell/N$ is the probability that every cycle of the uniformly
random permutation of $[N]$ contains at least one element of $[\ell]$; see Lov\'asz \cite{Lovasz},
Section 3, Exercise 6.

\begin{corollary} For all positive integers $\ell$, the identity 
\begin{equation} \label{A1k2} 
p_{A_1}(N,\ell;2)=(-1)^{N+\ell}N\binom{N}{\ell}^{-1}\sum_{i=0}^{N-\ell}(-1)^i\binom{N}{i}\frac{1}{i+\ell}
\end{equation}
holds.
\end{corollary}

\begin{proof}
Note that for $k=2$ we are able to replace the right-hand side of \eqref{Pr(keverywhere)} with a sum of just $\ell+1$
terms, which will allow us to determine compact formulas for moderate values of $\ell$. 
To do so we will need a certain binomial identity. Introduce
\begin{equation}\label{Snabdef}
S_{n,a,b}=\sum_{r=a+b}^n (-1)^r\frac{\binom{r-a}{b}}{\binom{n}{r}}.
\end{equation}
This function is relevant since \eqref{Pr(keverywhere)} is equivalent to
\begin{equation}\label{PA1=SN-1}
p_{A_1}(N,\ell;2)=
(-1)^{N-1}\binom{N}{\ell}^{-1}S_{N-1,0,\ell-1}.
\end{equation}
As we mentioned earlier 
\begin{equation}\label{SWZ}
S_{n,0,0}=(1+(-1)^n)\,\frac{n+1}{n+2},
\end{equation}
(\cite{Sury}, \cite{Stanley}, \cite{SWZ}), and the key element of the proofs was the identity 
\begin{equation}\label{int}
\binom{n}{r}^{-1} =(n+1)\int_0^1t^r(1-t)^{n-r}\,dt.
\end{equation}
In fact,  in \cite{Sury} the equation \eqref{int} was used to derive a sum-type formula, still with $n+1$ terms, for 
\[
\sum_{r=0}^n (-1)^r\frac{x^r}{\binom{n}{r}},
\]
that yielded \eqref{SWZ} via setting $x=1$. We also use \eqref{int} but avoid an intermediate 
sum with $n+1$ terms, and instead differentiate the resulting integral with respect to the 
parameter $x$. Here are the details. First define and evaluate $\cal S_{n,a,b}(x)$: for $a+b\le n$,
\begin{equation}\label{cals(nax)=}
\begin{aligned}
&\cal S_{n,a,b}(x):=\sum_{r=a+b}^n(-1)^r\frac{x^{r-a}}{\binom{n}{r}}\\
=&
(n+1)\!\int_0^1\!\left(\sum_{r=a+b}^n (-1)^r x^{r-a} t^r(1-t)^{n-r}\!\right) dt\\
=&(n+1)(-1)^{a+b}\int_0^1t^{a+b}(1-t)^{n-a-b}x^b\sum_{r=a+b}^n \left(-\frac{xt}{1-t}\right)^{r-a-b}\, dt\\
=&(n+1)(-1)^{a+b}\int_0^1\frac{x^b+(-1)^{n-a-b}x^{n-a+1}\left(\tfrac{t}{1-t}\right)^{n-a-b+1}}{1+\tfrac{xt}{1-t}}\\
&\times t^{a+b}(1-t)^{n-a-b}\,dt.
\end{aligned}
\end{equation}
The connection between $\cal S(n,a,x)$ and $S(n,a,b)$ is:
$S_{n,a,b}=\tfrac{1}{b!}\left.\frac{d^b \cal S_{n,a,b}(x)}{d x^b}\right|_{x=1}$.

To compute this derivative, we differentiate $b$ times the right-hand side of \eqref{cals(nax)=}
with respect to $x$  by carrying the operation inside the integral and then setting
$x=1$. So
\begin{equation}\label{calSn(b)}
\begin{aligned}
&\quad\quad\quad\left.\frac{d^b\cal S(n,a,x)}{dx^b} \right|_{x=1}\!=(-1)^{a+b}(n+1)\\
&\times \int_0^1\!\!
\frac{\partial^b}{\partial x^b}\,\frac{x^b+(-1)^{n-a-b}x^{n-a+1}\left(\tfrac{t}{1-t}\right)^{n-a-b+1}}{1+\tfrac{xt}{1-t}}
\biggr|_{x=1}\!\! t^{a+b}(1-t)^{n-a-b}\,dt.
\end{aligned}
\end{equation}
By $(uv)^{(b)}=\sum_j\binom{b}{j}u^{(j)}v^{(b-j)}$, the partial derivative at $x=1$ is 
\begin{multline*}
\sum_{j=0}^b\binom{b}{j}\left((b)_j+(-1)^{n-a-b}\left(\tfrac{t}{1-t}\right)^{n-a-b+1}(n-a+1)_j\,
\right)\\
\times (-1)^{b-j}\frac{(b-j)!}{\left(1+\tfrac{t}{1-t}\right)^{b-j+1}}\cdot\left(\frac{t}{1-t}\right)^{b-j}\\
=\sum_{j=0}^b(-1)^{b-j}\binom{b}{j}\biggl[b!\, t^{b-j}(1-t)+(-1)^{n-a-b}(b-j)!\,(n-a+1)_j
\frac{t^{n-a+1-j}}{(1-t)^{n-a-b}}\biggr]\\
=b!\left[(1-t)^{b+1} +\sum_{j=0}^b (-1)^{n-a-j}\binom{n-a+1}{j}\frac{t^{n-a+1-j}}{(1-t)^{n-a-b}}\right].
\end{multline*}
Plugging the last expression into \eqref{calSn(b)} and using \eqref{int} we obtain
\begin{equation}\label{sum(n,a,b)=expl}
\begin{aligned}
S_{n,a,b}=&\sum_{r=a+b}^n (-1)^r\frac{\binom{r-a}{b}}{\binom{n}{r}}
=(n+1)\biggl[\frac{(-1)^{a+b}}{(n+2+b)\binom{n+b+1}{a+b}}\\
&+\sum_{j=0}^b(-1)^{n+b-j}\binom{n-a+1}{j}\frac{1}{n+2+b-j}\biggr].
\end{aligned}
\end{equation}
For  large $n$, this formula is a significant improvement of the initial definition of $S_{n,a,b}$
if $b$ remains moderately valued. Using yet another identity 
\[
\sum_{j=0}^u(-1)^j\binom{u}{j}\frac{1}{v+j+1}=\frac{1}{(u+v+1)\binom{u+v}{v}},
\] from Sury et al. \cite{SWZ},
the equation \eqref{sum(n,a,b)=expl} is easily transformed into
\begin{equation}\label{sum(n,a,b)=expl,2}
S_{n,a,b}=(-1)^{a+b}(n+1)\sum_{i=0}^{n-a-b}(-1)^i\binom{n-a+1}{i}\frac{1}{i+a+b+1}.
\end{equation}
This alternative formula is efficient for the extreme case, when $n - a-b$ is moderately valued
as $n$ grows.

So, applying  the formulas \eqref{sum(n,a,b)=expl}, \eqref{sum(n,a,b)=expl,2} for $n=N-1$, $a=0$ and $b=\ell-1$, we obtain 
from \eqref{PA1=SN-1} that
\begin{multline}\label{p1Nell=expl}
p_{A_1}(N,\ell;2)=(-1)^{N-1}N\binom{N}{\ell}^{-1}\\
\times\biggl[\frac{(-1)^{\ell-1}}{(N+\ell)\binom{N+\ell-1}{\ell-1}}
+\sum_{j=0}^{\ell-1}(-1)^{N+\ell-j}\binom{N}{j}\frac{1}{N+\ell-j}\biggr]\\
=(-1)^{N+\ell}N\binom{N}{\ell}^{-1}\sum_{i=0}^{N-\ell}(-1)^i\binom{N}{i}\frac{1}{i+\ell},
\end{multline}
which was to be proved. 
\end{proof}

Note that the
two expressions in  formula \eqref{p1Nell=expl}, we just obtained for $ p_{A_1}(N,\ell;2)$,  can be efficiently computed for moderate $\ell$ and moderate $N-\ell$, respectively.

\begin{example} 
Using the first expression in \eqref{p1Nell=expl} we obtain
\[
p_{A_1}(N,1;2)=\left\{ \begin{array}{l@{\ }l}
\frac{2}{N+1} \,\,\hbox{ if $N$ is odd},\\
\\ 
0 \,\, \hbox{ if $N$ is even.}
\end{array} \right.
\] 
\end{example}

This is equivalent to the result already mentioned in Section \ref{cycles}, since $p_{A_1}(N,1;2)$
is indeed equal to the probability that $\sigma$ is a maximal cycle.
%


\subsection{Probability that the elements $1,\dots,\ell$ are in the same cycle of  $\sigma$} \label{A=A2}
This time $A=A_2=\{0,\ell\}$, so that
\begin{equation} \label{startforA2}
\sum_{a\in A_2}\binom{r}{a}w^a=1+\binom{r}{\ell}w^{\ell}.
\end{equation}

Our goal in this section is to prove the following theorem and its special case of $k=2$. 
\begin{theorem} \label{formulaforA2} For all integers $\ell \geq 2$, the identity
\begin{align*}
p_{A_2}(N,\ell;k)
=&\frac{1}{N!}\sum_{\la_1=1}^N(-1)^{k(\la^1-1)}\binom{N-1}{\la_1-1}^{-k+1}\sum_{s\text{ admissible}}\chi^{\la^*}(s)\\
=&\frac{1}{\ell}\binom{N}{\ell}^{-1}\sum_{\la_1}(-1)^{(k+1)(\la^1-1)}\binom{N-1}{\la_1-1}^{-k+1}\\
&\times \left\{1_{\{\la_1<N\}}\left[\binom{N-1}{\ell-1}-
\binom{N-\la_1-1}{\ell-1}\right]+1_{\{\la_1=N\}}\binom{N}{\ell}\right\}
\end{align*} holds.
\end{theorem}

\begin{proof} In this case, the computation is more involved than it was for $A_1$. 
Formula \eqref{startforA2} implies 
\begin{equation}\label{p_2(nu,ell)}
Q_{A_2}(\vec{\nu},\ell)=(N-\ell)!\,\ell!\,\, [w^{\ell}]\prod_r\frac{1}{\nu_r!}\left(\tfrac{1+\binom{r}{\ell}w^{\ell}}{r}\right)^{\nu_r}.
\end{equation}
So, using \eqref{chila*(nu)} and $\nu=\sum_r\nu_r$, we conclude that
\begin{multline}\label{sumcharadmsss2}
\sum_{s\text{ admissible}}\chi^{\la^*}(s)=(-1)^{\la^1}(N-\ell)!\,\ell!\\
\times 
[\xi^{\la_1}w^{\ell}]\,\,\frac{\xi}{1-\xi}\,\sum_{\vec\nu:\atop 1\nu_1+2\nu_2+\cdots=N}\prod_r\frac{1}{\nu_r!}\left(-(1-\xi^r)
\frac{1+\binom{r}{\ell}w^{\ell}}{r}\right)^{\nu_r}.
\end{multline}
Since $\sum_r r\nu_r=N$, the identity $\sum_r z^r/r =-\log(1-z)$, ($|z|<1$),  implies
that the second line expression in \eqref{sumcharadmsss2} equals
\begin{multline*}
[\xi^{\la_1}w^{\ell}x^N] \frac{\xi}{1-\xi}\,\sum_{\vec\nu\,\ge \bold 0}\prod_r\frac{(x^r)^{\nu_r}}{\nu_r!}\left(-(1-\xi^r)\frac{1+\binom{r}{\ell}w^{\ell}}{r}\right)^{\nu_r}\\
=[\xi^{\la_1}w^{\ell}x^N] \frac{\xi}{1-\xi}\,\prod_r\sum_{\nu_r\ge 0}\frac{1}{\nu_r!}
\left(-x^r(1-\xi^r)\frac{1+\binom{r}{\ell}w^{\ell}}{r}\right)^{\nu_r}\\
=[\xi^{\la_1}w^{\ell}x^N] \frac{\xi}{1-\xi}\,\prod_r\exp\left(-x^r(1-\xi^r)\frac{1+\binom{r}{\ell}w^{\ell}}{r}\right)\\
=[\xi^{\la_1}w^{\ell}x^N] \frac{\xi}{1-\xi}\,\exp\left(-\sum_{r\ge 1}x^r(1-\xi^r)
\frac{1+\binom{r}{\ell}w^{\ell}}{r}\right).
\end{multline*}
Here, using $\sum_{b\ge a}\binom{b}{a}z^b = \tfrac{z^a}{(1-z)^{a+1}}$,
\begin{multline*} 
\sum_{r\ge 1}x^r(1-\xi^r)\frac{1+\binom{r}{\ell}w^{\ell}}{r}\\
=-\log(1-x)+\log(1-x\xi) +\frac{w^{\ell}}{\ell}\sum_{r\ge 1}\binom{r-1}{\ell-1}\bigl(x^r-(x\xi)^r\bigr)\\
=\log\frac{1-x\xi}{1-x}+\frac{w^{\ell}}{\ell}\left(\frac{x^{\ell}}{(1-x)^{\ell}}-\frac{(x\xi)^{\ell}}{(1-x\xi)^{\ell}}
\right).
\end{multline*}
Therefore
\begin{multline*}
[w^{\ell}]\exp\left(-\sum_{r\ge 1}x^r(1-\xi^r)
\frac{1+\binom{r}{\ell}w^{\ell}}{r}\right)\\
=\frac{1-x}{1-x\xi}\,[w^{\ell}]\exp\left[-\frac{w^{\ell}}{\ell}\left(\frac{x^{\ell}}{(1-x)^{\ell}}-\frac{(x\xi)^{\ell}}{(1-x\xi)^{\ell}}
\right)\right]\\
=\frac{1}{\ell}\frac{1-x}{1-x\xi}\left(\frac{(x\xi)^{\ell}}{(1-x\xi)^{\ell}}-\frac{x^{\ell}}{(1-x)^{\ell}}\right).
\end{multline*}
Therefore the  expression in the second line of  \eqref{sumcharadmsss2} is equal to
\begin{multline*}
\frac{1}{\ell}\,[\xi^{\la_1}x^N]\,\frac{\xi}{1-\xi}\cdot\frac{1-x}{1-x\xi}\left(\frac{(x\xi)^{\ell}}{(1-x\xi)^{\ell}}-\frac{x^{\ell}}{(1-x)^{\ell}}\right)\\
=\frac{1}{\ell}\,[\xi^{\la_1}x^N]\,\left(\frac{1}{1-\xi}-\frac{1}{1-x\xi}\right)
\left(\frac{(x\xi)^{\ell}}{(1-x\xi)^{\ell}}-\frac{x^{\ell}}{(1-x)^{\ell}}\right)\\
=:\frac{1}{\ell}(T_1+T_2+T_3+T_4).
\end{multline*}
\bi
Here
\begin{equation}\label{T1=}
\begin{aligned}
T_1&=[\xi^{\la_1}x^N]\,\frac{1}{1-\xi}\cdot\frac{(x\xi)^{\ell}}{(1-x\xi)^{\ell}}\\
&=[\xi^{\la_1}]\,\frac{\xi^N}{1-\xi}\,[y^N]\frac{y^{\ell}}{(1-y)^{\ell}}
=1_{\{\la_1=N\}}\binom{N-1}{\ell-1};
\end{aligned}
\end{equation}
next
\begin{equation}\label{T2=}
\begin{aligned}
T_2&=-[\xi^{\la_1}x^N]\,\frac{1}{1-\xi}\cdot\frac{x^{\ell}}{(1-x)^{\ell}}\\
&=-[x^{N-\ell}]\,\frac{1}{(1-x)^{\ell}}=-\binom{N-1}{\ell-1};
\end{aligned}
\end{equation}
next
\begin{equation}\label{T3=}
\begin{aligned}
T_3&=-[\xi^{\la_1}x^N]\,\frac{(x\xi)^{\ell}}{(1-x\xi)^{\ell+1}}\\
&=-1_{\{\la_1=N\}}\,[y^{N-\ell}]\,\frac{1}{(1-y)^{\ell+1}}=-1_{\{\la_1=N\}}\binom{N}{\ell};
\end{aligned}
\end{equation}
and finally
\begin{equation}\label{T4=}
\begin{aligned}
T_4&=[\xi^{\la_1}x^N]\,\frac{1}{1-x\xi}\,\frac{x^{\ell}}{(1-x)^{\ell}}\\
&=[x^N]\,\frac{x^{\la_1+\ell}}{(1-x)^{\ell}}=[x^{N-\la_1-\ell}]\,\frac{1}{(1-x)^{\ell}}\\
&=1_{\{\la_1<N\}}\binom{N-\la_1-1}{\ell-1}.
\end{aligned}
\end{equation}
It follows from \eqref{T1=}, \eqref{T2=}, \eqref{T3=} and \eqref{T4=} that 
\begin{multline*}
\frac{1}{\ell}(T_1+T_2+T_3+T_4)\\
=-\frac{1}{\ell}\left\{1_{\{\la_1<N\}}\left[\binom{N-1}{\ell-1}-
\binom{N-\la_1-1}{\ell-1}\right]+1_{\{\la_1=N\}}\binom{N}{\ell}\right\}.
\end{multline*}
So \eqref{sumcharadmsss2} becomes
\begin{multline}\label{sumcharadmexpl}
\sum_{s\text{ admissible}}\chi^{\la^*}(s)=(-1)^{\la^1-1}(N-\ell)!\,\ell!\\
\times 
\frac{1}{\ell}\left\{1_{\{\la_1<N\}}\left[\binom{N-1}{\ell-1}-
\binom{N-\la_1-1}{\ell-1}\right]+1_{\{\la_1=N\}}\binom{N}{\ell}\right\}.
\end{multline}
Combining \eqref{sumcharadmexpl} and \eqref{Psigma(s)} we obtain the statement that was to be 
proved. 
\end{proof}

\begin{corollary} For all integers $\ell \geq 2$, we have
\begin{align} 
p_{A_2}(N,\ell;2)=&\frac{1}{\ell}-\frac{1}{(N+1)_2}\label{P2Nell,final2}\\
&+(-1)^{\ell+1}\binom{N-1}{\ell-1}^{-1}\sum_{i=0}^{N-\ell}(-1)^i\binom{N-1}{i}\frac{1}{i+\ell+1}.
\end{align}
\end{corollary}

\begin{proof} For $k=2$, introducing $r=N-\la_1$, we  have
\begin{equation}\label{P2Nell}
\begin{aligned}
p_{A_2}(N,\ell;2)=&\frac{1}{\ell}+\frac{1}{\ell}\binom{N}{\ell}^{-1}\\
&\times\sum_{r=1}^{N-1}(-1)^{r}\cdot
\binom{N-1}{r}^{-1}\left[\binom{N-1}{\ell-1}-
\binom{r-1}{\ell-1}\right].
\end{aligned}
\end{equation}
By \eqref{Snabdef}, the last sum is the linear combination of $S_{N-1,0,0}-1$ and $S_{N-1,1,\ell-1}$. According to \eqref{sum(n,a,b)=expl} and \eqref{sum(n,a,b)=expl,2}, we have
\begin{align*}
S_{N-1,0,0}=&\bigl[1+(-1)^{N-1}\bigr]\frac{N}{N+1},\\
S_{N-1,1,\ell-1}=&(-1)^{\ell}\biggl[\binom{N+\ell}{\ell}^{-1}+N\sum_{j=0}^{\ell-1}(-1)^{N-j}\binom{N-1}{j}\frac{1}{N+\ell-j}\biggr]\\
=&(-1)^{\ell}N\sum_{i=0}^{N-1-\ell}(-1)^i\binom{N-1}{i}\frac{1}{i+\ell+1}.
\end{align*}
Plugging these expressions into \eqref{P2Nell}, we obtain after simple algebra
\begin{align}
p_{A_2}(N,\ell;2)=&\frac{1}{\ell}+\left[\frac{1+(-1)^{N-1}}{N+1}-\frac{1}{N}\right]\notag\\
&+\frac{(-1)^{\ell+1}}{\ell\binom{N}{\ell}}\left[\binom{N+\ell}{\ell}^{-1}+N\sum_{j=0}^{\ell-1}
(-1)^{N-j}\binom{N-1}{j}\frac{1}{N+\ell-j}\right]\notag\\
=&\frac{1}{\ell}-\frac{1}{(N+1)_2}\label{P2Nell,final}\\
&+\frac{(-1)^{\ell+1}}{\ell\binom{N}{\ell}}\left[\binom{N+\ell}{\ell}^{-1}+N\sum_{j=0}^{\ell-2}
(-1)^{N-j}\binom{N-1}{j}\frac{1}{N+\ell-j}\right]\notag\\
=&\frac{1}{\ell}-\frac{1}{(N+1)_2}\label{P2Nell,final2}\\
&+(-1)^{\ell+1}\binom{N-1}{\ell-1}^{-1}\sum_{i=0}^{N-\ell}(-1)^i\binom{N-1}{i}\frac{1}{i+\ell+1},
\notag
\end{align} as  claimed. 
\end{proof}

The equivalent  formulas \eqref{P2Nell,final} and \eqref{P2Nell,final2} are computationally efficient
for moderate $\ell$ and moderate $N-\ell$ respectively.
In particular, plugging $\ell=2,3$ into \eqref{P2Nell,final} and simplifying, we recover Stanley's results, \cite{Stanley2}.

\section{The probability that $\sigma$ separates the disjoint sets ${\cal S}_1,\dots,{\cal S}_t$} \label{separate}
Let $\ell_j=|{\cal S}_j|$, $1\le j\le t$, $\ell=\sum_j\ell_j$. Introduce $p(N,\vec\ell;k)$, the probability that the permutation $\sigma$ separates the sets ${\cal S}_1,\dots,
{\cal S}_t$, meaning that no cycle of $\sigma$ contains a pair of elements from two distinct sets ${\cal S}_i$ and ${\cal S}_j$. 
 Bernardi et al. \cite{Bernardi} were able to derive a striking formula for $p(N,\vec\ell;2)$:
\begin{equation}\label{bern}
p(N,\vec\ell;2)=\frac{(N-\ell)!\prod_j\ell_j!}{(N+t)(N-1)!}
\left[\frac{(-1)^{N+\ell}\binom{N-1}{t-2}}{\binom{N+\ell}{\ell-t}}+\sum_{j=0}^{\ell-t}
\frac{(-1)^j\binom{\ell-t}{j}\binom{N+j+1}{\ell}}{\binom{N+t+j}{j}}\right],
\end{equation}
which is a sum of $\ell-t+2$ terms. Remarkably, $\prod_j\ell_j!$ aside, the rest of this expression 
does not depend on the individual $\ell_j$. The equation \eqref{bern} is very efficient for 
values of $\ell$, $t$ relatively small compared to $N$.

In this section first we apply our approach to obtain a formula for this probability for a general $k\ge 2$.  Similarly to $p(N,\vec\ell;2)$, it is of a form $\prod_j\ell_j!$ times an expression that
depends on $\ell$, but not on individual $\ell_j$. 

\begin{lemma}\label{generalk} Introduce
\[
K(N,\ell,t;r)=\bigl[\xi^{r-\ell+t}\eta^{N-\ell}\bigr] \left(\frac{1-\xi}{1-\eta}\right)^{t-1}\!\!(1-\xi\eta)^{-\ell-1},
\]
and define $\alpha_k(N,t)=t-1$ if $k$ is odd, and $\alpha_k(N,t)=N+t$ if $k$ is even. Then 
\begin{equation}\label{Pr(separ)}
p(N,\vec\ell;k)
\frac{(-1)^{\alpha_k(N,t)}\prod_j\ell_j!}{(N)_{\ell}}
\sum_{r=\ell-t}^{N-1}
(-1)^{(k+1)r} \binom{N-1}{r}^{-k+1}\!\!K(N,\ell,t;r).
\end{equation} 
\end{lemma}
\noindent The formula \eqref{Pr(separ)} is computationally efficient for $\ell - t$ close to $N$.
\begin{proof}  Let $Q(\vec\nu,\vec\ell)$ denote the total number
 of permutations of $[N]$ with
cycle counts $\vec\nu=(\nu_1,\nu_2,\dots)$ that separate ${\cal S}_1,\dots,{\cal S}_t$.
Each cycle of such a permutation either does not contain any element of $\cup_j {\cal S}_j$, or
contains some of the elements of exactly one set ${\cal S}_j$. Since $\left|[N]\cup_j{\cal S}_j\right|=N-\ell$, denoting $\prod_jw_j^{\ell_j}=\vec w\,^{\vec\ell}$,
analogously to \eqref{p(nu,ell)} we have
\begin{equation}\label{QNboldell}
\begin{aligned}
\frac{Q(\vec\nu,\vec\ell)}
{(N-\ell)!\,\prod_j\ell_j!}=&\,[y^{N-\ell}\vec w\,^{\vec\ell}]\left[\prod_r\frac{1}{\nu_r!}\left(\frac{y^r+\sum_{j=1}^t\sum_{a>0}\binom{r}{a}w_j^ay^{r-a}}{r}\right)^{\nu_r}\right]\\
=&\,[y^{N-\ell}\vec w\,^{\vec\ell}]\left[\prod_r\frac{1}{\nu_r!}\left(\frac{-(t-1)y^r+
\sum_{j=1}^t(w_j+y)^r}{r}\right)^{\nu_r}\right].
\end{aligned}
\end{equation}
Using \eqref{chila*(nu)} and
\eqref{QNboldell}, we obtain
\begin{multline}\label{sumcharadmsss}
\sum_{s\text{ admissible}}\chi^{\la^*}(s)=(-1)^{\la^1}(N-\ell)!\, \prod_j \ell_j!\\
\times 
[\xi^{\la_1}y^{N-\ell}\vec w\,^{\vec\ell}]\,\,\frac{\xi}{1-\xi}\,\sum_{\vec\nu}\prod_r\frac{1}{\nu_r!}\left(-\frac{(1-\xi^r)\bigl(-(t-1)y^r+\sum_j(w_j+y)^r\bigr)}{r}\right)^{\nu_r},
\end{multline}
the sum being for $\vec\nu\,\ge \bold 0$ with $\sum_r r\nu_r=N$.  So the  expression in the second line of
 \eqref{sumcharadmsss} equals 
\begin{align*}
&[\xi^{\la_1}x^N y^{N-\ell}\vec w\,^{\vec\ell}] \frac{\xi}{1-\xi}\,\sum_{\vec\nu\,\ge \bold 0}\prod_r\frac{(x^r)^{\nu_r}}{\nu_r!}\\
&\times\left(-\frac{(1-\xi^r)\bigl[-(t-1)y^r+\sum_j(w_j+y)^r\bigr]}{r}\right)^{\nu_r}\\
=&\,[\xi^{\la_1}x^Ny^{N-\ell}\vec w\,^{\vec\ell}] \frac{\xi}{1-\xi}\,\prod_r\exp\left(-\frac{x^r(1-\xi^r)
\bigl[-(t-1)y^r+\sum_j(w_j+y)^r\bigr]}{r}\right)\\
=&\,[\xi^{\la_1}x^Ny^{N-\ell}] \frac{\xi}{1-\xi}\,\left(\frac{1-\xi xy}{1-xy}\right)^{t-1}
[\vec w\,^{\vec\ell}]\prod_j\frac{1-x(w_j+y)}{1-\xi x(w_j+y)}\\
=&\,[\xi^{\la_1}x^Ny^{N-\ell}] \frac{\xi}{1-\xi}\,\left(\frac{1-\xi xy}{1-xy}\right)^{t-1}\left(\frac{1-xy}{1-\xi xy}\right)^{t}
\cdot\prod_j [w_j^{\ell_j}]\frac{1-\tfrac{xw_j}{1-xy}}{1-\tfrac{\xi x w_j}{1-\xi xy}}\\
=&\,[\xi^{\la_1}x^Ny^{N-\ell}] \frac{\xi}{1-\xi}\,\frac{1-xy}{1-\xi xy}
\prod_j\left[\left(\frac{\xi x}{1-\xi xy}\right)^{\ell_j}-\frac{x}{1-xy}\left(\frac{\xi x}{1-\xi xy}\right)^{\ell_j-1}\right]\\
=&\,[\xi^{\la_1}x^Ny^{N-\ell}] \frac{\xi}{1-\xi}\,\frac{1-xy}{1-\xi xy}\left(\frac{\xi x}{1-\xi xy}\right)
^{\ell-t}\left(\frac{(\xi-1)x}{(1-\xi xy)(1-xy)}\right)^t\\
=&(-1)^{t}\,[\xi^{\la_1-1-(\ell-t)}x^{N-\ell}y^{N-\ell}](1-\xi)^{t-1}(1-xy)^{-t+1}(1-\xi xy)^{-\ell-1}\\
=:&\,(-1)^{t} K(N,\ell,t;\la_1-1).
\end{align*}
Thus, $\xi$ aside, we need to extract a coefficient of $(xy)^{N-\ell}$ from a power series of $xy$.  
So
\begin{equation}\label{K=}
\begin{aligned}
K(N,&\ell,t;r):=[\xi^{r-\ell+t}z^{N-\ell}]\,(1-\xi)^{t-1}(1-z)^{-t+1}(1-\xi z)^{-\ell-1}\\
=&\sum_j(-1)^{r-\Delta-j}\binom{\ell+j}{j}\binom{t-1}{r-\Delta-j}\binom{N-\Delta-j-2}{t-2},
\end{aligned}
\end{equation}
where we set $\Delta=\ell-t$. Obviously $K(N,\ell,t;r)=0$ for $r<\ell-t$, and less obviously for $r\ge N$. Indeed
\begin{multline}\label{r>N}
[z^{N-\ell}]\,(1-\xi)^{t-1}(1-z)^{-t+1}(1-\xi z)^{-\ell+1}\\
=\sum_{j\le N-\ell}(-1)^{N-\ell-j}\binom{-t+1}{N-\ell-j}\,[z^j](1-\xi)^{t-1}(1-\xi z)^{-\ell-1},
\end{multline}
and the $[z^j]$-factor is a polynomial of $\xi$ of degree $t-1+j\le t-1+N-\ell< r-\ell+t$ if $r\ge N$.

Combining this with  equation \eqref{Psigma(s)}, and $\la_1+\la^{1}=N+1$, we obtain the statement that was
to be proved.
\end{proof}

The sum in \eqref{Pr(separ)} depends only on $\ell$ and $t$, rather than the individual $\ell_1,\dots,\ell_t$, and 
$K(N,\ell,t,r)$ is given by each of two lines in \eqref{K=}. In particular,
\[
K(N,N,t;r)=[\xi^{r-N+t}](1-\xi)^{t-1}=(-1)^{r-N+t}\binom{t-1}{r-N+t}.
\]
Let $\ell=\sum_j\ell_j=N$.  Introducing $\beta_k(N)=N-1$ for $k$ odd, $\beta_k(N)=0$ for $k$ even,  equation  \eqref{Pr(separ)} becomes
\[
p(N,\vec\ell;k)=
\frac{(-1)^{\beta_k(N)}\prod_j\ell_j!}{(N)_{\ell}}
\sum_{r=N-t}^{N-1}
(-1)^{kr} \binom{N-1}{r}^{-k+1}\binom{t-1}{r-N+t},
\]
an alternating sum of $t$ terms. For $t=N$, $p(N,\vec\ell;k)=\pr(\sigma=\text{id})$; the resulting
formula agrees with \eqref{Pr(id)=}, since for $k$ odd and $N$ even the sum over $r\in [0,N-1]$
is zero.\\

\subsection{When $k=2$}
From now on we focus on $k=2$, and general $\vec\ell$. We begin with a relatively compact
formula that represents $p(N,\vec\ell;2)$ as a composition of integration operation and coefficient extraction operation.

\begin{theorem} The identity
\begin{equation}\label{Psep,k=2,equation}
\begin{aligned}
p(N,\vec\ell;2)=&\frac{(-1)^{N+\ell}N\prod_j\ell_j!}{(N)_{\ell}}\\
&\times [z^{N-\ell}] (1-z)^{-t+1}
\int_0^1\frac{(1-u)^{N+1}u^{\ell-t}}{(1-u+zu)^{\ell+1}}\,du.
\end{aligned}
\end{equation}
holds. 
\end{theorem} 

\begin{proof}
As $k=2$, equation \eqref{Pr(separ)}
becomes
\begin{align}
p(N,\vec\ell;2)
=&\frac{(-1)^{N+t}\prod_j\ell_j!}{(N)_{\ell}}\sum_{r=\ell-t}^{N-1}(-1)^{r} \binom{N-1}{r}^{-1}\!\!K(N,\ell,t;r),
\label{Pr(separ),k=2}\\
K(N,\ell,t;r):=&\,[\xi^{r-\ell+t}z^{N-\ell}]\,(1-\xi)^{t-1}(1-z)^{-t+1}(1-\xi z)^{-\ell-1}.\notag
\end{align}
In \eqref{Pr(separ),k=2} we can extend the summation to $r\in [\ell-t,\infty)$, since
$K(N,\ell,t;r)=0$ for $r\ge N$.

Let us evaluate the sum in \eqref{Pr(separ),k=2} halfway, i.e. dropping $(1-z)^{-t+1}$ and postponing the extraction of the
coefficient by $z^{N-\ell}$ till the next step. Using \eqref{int}, {\it and\/}  the observation above  to replace $N-1$ with $\infty$, we reduce the halfway sum 
 to
\begin{multline}\label{sumreduced}
N\sum_{r=\ell-t}^{\infty}(-1)^r [\xi^{r-\ell+t}]\,\frac{(1-\xi)^{t-1}}{(1-\xi z)^{\ell+1}}
\int_0^1u^r(1-u)^{N-1-r}\,du\\
=N\int_0^1(1-u)^{N-1}\left(\sum_{r=\ell-t}^{\infty}\left(-\frac{u}{1-u}\right)^r\,[\xi^{r-\ell+t}]
\frac{(1-\xi)^{t-1}}{(1-\xi z)^{\ell+1}}
\right)\,du\\
=N\int_0^1(1-u)^{N-1}\left(-\frac{u}{1-u}\right)^{\ell-t}\left(\sum_{r=\ell-t}^{\infty}[\xi^{r-\ell+t}]
\frac{(1+\xi\tfrac{u}{1-u})^{t-1}}{(1+\xi z\tfrac{u}{1-u})^{\ell+1}}\right)\,du\\
=N\int_0^1(1-u)^{N-1}\left(-\frac{u}{1-u}\right)^{\ell-t}\left(\sum_{r=\ell-t}^{\infty}[\xi^{r-\ell+t}]
\frac{(1+\xi\tfrac{u}{1-u})^{t-1}}{(1+\xi z\tfrac{u}{1-u})^{\ell+1}}\right)\,du\\
=N\int_0^1(1-u)^{N-1}\left(-\frac{u}{1-u}\right)^{\ell-t}\left.\frac{(1+\xi\tfrac{u}{1-u})^{t-1}}{(1+\xi z\tfrac{u}{1-u})^{\ell+1}}\right|_{\xi=1}\,du\\
=(-1)^{\ell-t}N\int_0^1\frac{(1-u)^{N+1}u^{\ell-t}}{(1-u+zu)^{\ell+1}}\,du;
\end{multline}
(in the fifth line we used $\sum_{r\ge 0} [\xi^r] f(\xi) =f(1)$ for the series $f(\xi)=\sum_{r\ge 0} a_r\xi^r$).
So \eqref{Pr(separ),k=2} is transformed into
\begin{equation}\label{Psep,k=2,transform}
\begin{aligned}
p(N,\vec\ell;2)=&\frac{(-1)^{N+\ell}N\prod_j\ell_j!}{(N)_{\ell}}\\
&\times [z^{N-\ell}] (1-z)^{-t+1}
\int_0^1\frac{(1-u)^{N+1}u^{\ell-t}}{(1-u+zu)^{\ell+1}}\,du ,
\end{aligned}
\end{equation} which is the formula that was to be proved. 
\end{proof}

\begin{corollary} For $\ell=N$ the formula \eqref{Psep,k=2,transform} yields
\begin{equation}\label{Psep,k=2,ell=N}
p(N,\vec\ell;2)=\frac{N\prod_j\ell_j!}{N!}\,\int_0^1u^{N-t}\,du=
\frac{\prod_j\ell_j!}{(N-1)!(N-t+1)}.
\end{equation}
\end{corollary}
\noindent To compare, the separation probability for the uniformly random permutation of $[N]$ is 
$\prod_j\ell_j!/N!$.

For $\ell_1=\cdots=\ell_{t-1}=1$, $\ell_t=N-t+1$, ($2\le t\le N$),  $p(N,\vec\ell;2)$ is the probability that all elements of a given subset of cardinality $t-1$ are fixed points of $\sigma^{(2)}$; the number of such subsets is $\binom{N}{t-1}$. Furthermore the probability that all the elements of $[N]$
are fixed, i.e. $\sigma^{(2)}=\text{id}$, is  $\tfrac{1}{(N-1)!}$, see \eqref{P(sigma2isid)}. So using the inclusion-exclusion formula, we obtain:
\[
\pr(\sigma^{(2)}\text{ is a derangement})=N\sum_{\tau=0}^{N-1}\frac{(-1)^{\tau}}{(N-\tau)\tau!}
+\frac{(-1)^N}{(N-1)!}.
\]
For comparison, the probability that the uniformly random permutation of $[N]$ is
a derangement equals  $\sum_{\tau=0}^N(-1)^{\tau}\tfrac{1}{\tau!}$.

More generally,
\begin{equation}\label{Psep,k=2,N-ellsmall}
p(N,\vec\ell;2)=\frac{N\prod_j\ell_j!}{(N)_{\ell}}\sum_{k\le N-\ell}(-1)^k
\frac{\binom{t+k-2}{t-2}\binom{N-k}{\ell}}{(N-t+1)\binom{N-t}{k}},
\end{equation}
an equation computationally efficient for moderate $N-\ell$, but progressively less useful for larger values of $N-\ell$.

\subsection{An alternative formula deduced by the WZ-method}
In this section, we will show that equation \eqref{Psep,k=2,transform} can be transformed so that
extraction of the coefficient of $z^{N-\ell}$ will lead to a sum with $\ell-t+2$ number of terms,
close in appearance to the formula \eqref{bern} by Bernardi et al. 

Clearly it is the outside factor $(1-z)^{-t+1}$ that causes the number of summands in 
\eqref{Psep,k=2,N-ellsmall} grow indefinitely with $N$.  To get rid
of $(1-z)^{-t+1}$, we resort to repeated integration by parts  of the integral, denote it $I(z)$, with each step producing the outside factor $1-z$. However the factor $u^{\ell-t}$ in the integrand of $I(z)$ would have made the integration process unwieldy; so we apply it instead to $K_1(z)$, where
\[
K_{\nu}(z):=\int_0^1\frac{(1-u)^{N+\nu}}{(1-u+zu)^{t+\nu}}\,du,
\]
because 
\begin{equation}\label{I(through)K1}
I(z)=\frac{(-1)^{\ell-t}}{(t+1)^{(\ell-t)}}\,\frac{d^{\ell-t}K_1(z)}{dz^{\ell-t}}.
\end{equation}
One integration by parts leads to 
\begin{align*}
K_1(z)=&\frac{1}{N+2}+\frac{(t+1)(1-z)}{N+2}\int_0^1\frac{(1-u)^{N+2}}{(1-u+zu)^{t+2}}\,du\\
=&\frac{1}{N+2}+\frac{(t+1)(1-z)}{N+2}K_2(z).
\end{align*}
After $\ell-1$ integrations by parts, we get
\begin{equation*}
K_1(z)=\sum_{j=1}^{\ell-1}\frac{(t+1)^{(j-1)}}{(N+2)^{(j)}}(1-z)^{j-1}
+\frac{(t+1)^{(\ell-1)}}{(N+2)^{(\ell-1)}}\,(1-z)^{\ell-1} K_{\ell}(z).
\end{equation*}
So, using \eqref{I(through)K1} and
\[
\frac{d^{\ell-t}\bigl[(1-z)^{\ell-1}K_{\ell}\bigr]}{dz^{\ell-t}}=\sum_{\mu=0}^{\ell-t}
(-1)^{\mu}\binom{\ell-t}{\mu}(\ell-1)_{\mu}(1-z)^{\ell-1-\mu}\frac{d^{\ell-t-\mu} K_{\ell}}{dz^{\ell-t-\mu}},
\]
we obtain
\begin{align*}
&\frac{(1-z)^{-t+1}I(z)}{\tfrac{(-1)^{\ell-t}}{(t+1)^{(\ell-t)}}}
=\,(-1)^{\ell-t}\sum_{j=1}^{\ell-1}\frac{(t+1)^{(j-1)}(j-1)_{\ell-t}}{(N+2)^{(j)}}(1-z)^{j-\ell}\\
&+\frac{(t+1)^{(\ell-1)}}{(N+2)^{(\ell-1)}}\sum_{\mu=0}^{\ell-t}(-1)^{\mu}\binom{\ell-t}{\mu}(\ell-1)_{\mu}(1-z)^{\ell-t-\mu}
\,\frac{d^{\ell-t-\mu} K_{\ell}(z)}{dz^{\ell-t-\mu}}.
\\
\end{align*}
It remains to extract the coefficient of $[z^{N-\ell}]$ in the right-hand side expression. First,
\[
[z^{N-\ell}](1-z)^{j-\ell}=(-1)^{N-\ell}\binom{j-\ell}{N-\ell}.
\]
Next, for every $r\ge 0$,
\begin{align*}
&[z^r]\frac{d^{\ell-t-\mu} K_{\ell}}{dz^{\ell-t-\mu}}= (-1)^{\ell-t-\mu}(t+\ell)^{(\ell-t-\mu)}[z^r]
\int_0^1\frac{(1-u)^{N+\ell}u^{\ell-t-\mu}}{(1-u+zu)^{2\ell-\mu}}\,du\\
=&\,(-1)^{\ell-t-\mu}(t+\ell)^{(\ell-t-\mu)}\binom{-2\ell+\mu}{r} \int_0^1(1-u)^{N-\ell+\mu-r}u^{\ell-t-\mu+r}\,du\\
=&\,(-1)^{\ell-t-\mu}\frac{(t+\ell)^{(\ell-t-\mu)}\binom{-2\ell+\mu}{r}}
{(N-t+1)\binom{N-t}{\ell-t-\mu+r}}.
\end{align*}
So
\begin{equation}\label{k-sum}
\begin{aligned}
&[z^{N-\ell}]\,\left\{(1-z)^{\ell-t-\mu}\,\frac{d^{\ell-t-\mu}\, K_{\ell}}{dz^{\ell-t-\mu}}\right\}\\
=&\sum_{k\le \ell-t-\mu} \left\{[z^k] (1-z)^{\ell-t-\mu}\right\}\,\left\{[z^{N-\ell-k}]\,\frac{d^{\ell-t-\mu}\, K_{\ell}}{dz^{\ell-t-\mu}}\right\}\\
=&\sum_{k\le \ell-t-\mu}\!\! \!(-1)^k\binom{\ell-t-\mu}{k}\!\!
\left.(-1)^{\ell-t-\mu}\frac{(t+\ell)^{(\ell-t-\mu)}\binom{-2\ell+\mu}{r}}
{(N-t+1)\binom{N-t}{\ell-t-\mu+r}}\right|_{r=N-\ell-k}.
\end{aligned}
\end{equation}
Collecting the pieces,
\begin{align*}
&\frac{[z^{N-\ell}](1-z)^{-t+1}I(z)}{\tfrac{(-1)^{\ell-t}}{(t+1)^{(\ell-t)}}}\\
&\quad=(-1)^{N-t}\sum_{j=1}^{\ell-1}\frac{(t+1)^{(j-1)}(j-1)_{\ell-t}}{(N+2)^{(j)}}\binom{j-\ell}{N-\ell}\\
&\quad+(-1)^{\ell-t}\frac{(t+1)^{(\ell-1)}}{(N+2)^{(\ell-1)}}\sum_{\mu=0}^{\ell-t}\binom{\ell-t}{\mu}(\ell-1)_{\mu}(t+\ell)^{(\ell-t-\mu)}\\
&\quad\times\sum_{k\le \ell-t-\mu} (-1)^k\binom{\ell-t-\mu}{k}
\frac{\binom{-2\ell+\mu}{N-\ell-k}}
{(N-t+1)\binom{N-t}{\mu+k}}.
\end{align*}
So, since 
\[
\binom{-a}{b}=(-1)^b\binom{a+b-1}{a-1},\quad\frac{(t+1)^{(\ell-1)}(t+\ell)^{(\ell-t-\mu)}}{(t+1)^{(\ell-t)}}
=\frac{(2\ell-\mu-1)!}{\ell!},
\]
equation \eqref{Psep,k=2,transform} becomes
\begin{equation}\label{pnell2=final?}
\begin{aligned}
p(N\!,&\vec\ell;2)=\frac{N\prod_j\ell_j!}{(N)_{\ell}}\\
&\times
\left[(-1)^{N+\ell}\sum_{j=1}^{\ell-1}\frac{(t+1)^{(j-1)}(j-1)_{\ell-t}}{(t+1)^{(\ell-t)}(N+2)^{(j)}}
\binom{N-j-1}{\ell-j-1}\right.\\
&+\frac{1}{\ell! (N+2)^{(\ell-1)}(N-t+1)}\\
&\times\left.\sum_{\mu=0}^{\ell-t}\binom{\ell-t}{\mu}(\ell-1)_{\mu}
\sum_{\nu=\mu}^{\ell-t} \binom{\ell-t-\mu}{\ell-t-\nu}
\frac{(N+\ell-\nu-1)_{2\ell-\mu-1}}
{\binom{N-t}{\nu}}\right];
\end{aligned}
\end{equation}
$\nu$ in the bottom sum comes from substitution $\nu=k+\mu$ in \eqref{k-sum}. Changing the
 order of summation, the double sum above equals
\begin{multline}\label{doublesum=}
\frac{(\ell-t)!}{(N-\ell)!}\sum_{\nu=0}^{\ell-t}\frac{(N+\ell-\nu-1)!}{(\ell-t-\nu)!}\frac{1}{\binom{N-t}{\nu}}
\sum_{\mu=0}^{\nu}\binom{\ell-1}{\mu}\binom{N-\ell}{\nu-\mu}\\
=\frac{(\ell-t)!}{(N-\ell)!}\sum_{\nu=0}^{\ell-t}\frac{(N+\ell-\nu-1)!}{(\ell-t-\nu)!}\,\frac{\binom{N-1}{\nu}}
{\binom{N-t}{\nu}}.
\end{multline}
Let $\Sigma(N,\ell,t)$ denote the top, ordinary, sum in \eqref{pnell2=final?}.

\begin{lemma} The identity   
\begin{equation}\label{ell-t=1}
\Sigma(N,\ell,t)=\frac{(N-1)_{t-2}\,(\ell-t)!}{(t-2)!(N+t)^{(\ell-t+1)}}.
\end{equation}
holds. 
\end{lemma} 

\begin{proof} 
We confirmed this conjecture via the powerful Wilf-Zeilberger algorithm, see Nemes et al. \cite{WZ1},
Wilf and Zeilberger \cite {WZ2}.
  Given $\Delta\ge 0$, introduce a function of $t\ge 2$, defined by
\[
S(t)=\sum_{j=1}^{t-1+\Delta}\frac{(t+1)^{(j-1)}(j-1)_{\Delta}}{(t+1)^{(\Delta)}(N+2)^{(j)}}\binom{N-j-1}
{t+\Delta-j-1}.
\]
The non-zero summands are those for $j\in [\Delta+1, t-1+\Delta]$.  We can extend summation to $j\in [1,\infty)$, since the last binomial is zero for $j\ge t+\Delta$. We need to show that
\begin{equation}\label{S=S*}
S(t)=S^*(t):=\frac{(N-1)_{t-2}\Delta!}{(t-2)!(N+t)^{(\Delta+1)}}.
\end{equation}
To do so, first we compute
\begin{align*}
&\,\quad\frac{S^*(t)}{S^*(t-1)}=\frac{\beta(t)}{\alpha(t)},\\
\alpha(t):=(t-2)&(N+t+\Delta),\quad \beta(t):=(N-t+2)(N+t-1).
\end{align*}
Next, let $F(t,j)$ stand for the $j$-term in the series $S(t)$. Introduce the ``partner'' sequence $G(t,j)$ (which again for each $t$ is $0$ for all but finitely many $j$) such that
\begin{equation}\label{Grecur}
G(t,j)-G(t,j-1)=\alpha(t)F(t,j) -\beta(t)F(t-1,j),\quad j\ge \Delta+1,
\end{equation}
and $G(t,\Delta)=0$. 

The equation \eqref{S=S*} will be proved if we demonstrate that
$G(t,j)=0$ for $j$ large enough. 

Computation by Maple shows that 
\begin{align*}
G(t,\Delta+1)=&-\frac{(\Delta+1)!(\Delta+2t-2)}{(N+2)^{(\Delta+1)}}\binom{N-\Delta-2}{t-3},\\
G(t,\Delta+2)=&-\frac{(\Delta+2)!(\Delta+2t-2)(t+\Delta+1)}{(N+2)^{(\Delta+2)}}\binom{N-\Delta-3}{t-4},\\
G(t,\Delta+3)=&-\frac{(\Delta+3)!(\Delta +2t-2)(t+\Delta+2)_2}{2(N+2)^{(\Delta+3)}}
\binom{N-\Delta-4}{t-5}.
\end{align*}
The evidence is unmistakable: it must be true that for all $u\ge 1$
\begin{equation}\label{G(t,d+u)=}
G(t,\Delta+u)=-\frac{(\Delta+u)!(\Delta+2t-2)\binom{t+\Delta+u-1}{u-1}}{(N+2)^{(\Delta+u)}}
\binom{N-\Delta-u-1}{t-u-2}.
\end{equation}
Sure enough, the inductive step based on the recurrence \eqref{Grecur} is easily carried out with a 
guided assistance of Maple. It remains to notice that the last binomial coefficient is zero for $u>t-2$.
\end{proof}

Now we are in a position to announce the main result of this section. 
\begin{theorem} \label{k=2separate}  The identity
\begin{multline}\label{pnell2=final???}
p(N,\vec\ell;2)=\frac{(N-\ell)!\,\prod_j\ell_j!}{(N-1)!\,(N+t)}
\left[(-1)^{N+\ell}\frac{\binom{N-1}{t-2}}{\binom{N+\ell}{\ell-t}}\right.\\
+\frac{(N+t)(N+1)_{\ell+1}}{(N-t+1)(N+\ell)!\,(\ell)_t}
\left.\sum_{\nu=0}^{\ell-t}\frac{(N+\ell-\nu-1)!(N-1)_{\nu}}{(\ell-t-\nu)!(N-t)_{\nu}}\right].
\end{multline} holds.
\end{theorem} 

 \begin{proof} Combining \eqref{doublesum=}
and \eqref{ell-t=1}, formula
\eqref{pnell2=final?} simplifies  to our claim.
\end{proof}

The outside factor and the first inside term of \eqref{pnell2=final???}  are exactly those in  \eqref{bern} by Bernardi et al.
The second inside term,  a sum of $\ell-t+1$ terms, times  $\tfrac{(N+t)(N+1)_{\ell+1}}{(N-t+1)(N+\ell)!\,(\ell)_t}$, is quite different in appearance from its counterpart 
in \eqref{bern}. For $\ell-t\le 5$, Maple confirms that the rational functions given by the sums are
identical; we did not try to prove equality in general.

\section{Probability that  $\sigma$ blocks the 
elements of $[\ell]$}\label{isolated}

We say that the elements of $[\ell]$ are blocked in a permutation $s$ of $[N]$
if in every cycle of $s$  (1) no two elements of $[\ell]$ are neighbors, and (2) each element from $[\ell]$ has a neighbor from $[N]\setminus [\ell]$. 

Let $p(N,\ell;k)$ denote the probability of the event that $\sigma$ blocks the elements of $[\ell]$. In this final section, 
we are going to prove the following theorem. 

\begin{theorem} \label{twotermformula}
For all positive integers $\ell$ and $k$, the formula 
\begin{equation}\label{p4Nellk=}
p(N,\ell;k)=\frac{\binom{N-\ell}{\ell}}{\binom{N}{\ell}}
+(-1)^{k+1}\frac{\binom{N-\ell-1}{\ell-1}}{(N-1)^{k-1}\binom{N}{\ell}}.
\end{equation}
holds.
\end{theorem}

\begin{proof}  Let us
begin again with $Q(\vec\nu,\ell)$, the total number of permutations with cycle counts
$\vec\nu$ such that the elements of $[\ell]$ are blocked.
To evaluate $Q(\vec{\nu},\ell)$, introduce the non-negative integers $a_{r,j}$, $b_{r,j}$ that
stand for the generic numbers of elements from $[\ell]$ and $[N]\setminus [\ell]$  in the $j$-th cycle of
length $r$, $(j\le \nu_r)$. Then 
\begin{align}
 a_{r,j}&+b_{r,j}=r,\label{restr1}\\
 &b_{r,j}>0,\label{restr2} \\
\sum_{r,\,j\le\nu_r}a_{r,j}=&\ell,\quad \sum_{r,\,j\le\nu_r}b_{r,j}=N-\ell.\label{sums}
\end{align}
For $a_{r,j}>0$, the number of admissible cycles with parameters $a_{r,j}$, $b_{r,j}$ is
\begin{equation}\label{admcycles}
c(a_{r,j},b_{r,j}):= (a_{r,j}-1)!\,b_{r,j}!\,\binom{b_{r,j}-1}{a_{r,j}-1}=(b_{r,j}-1)!a_{r,j}!\binom{b_{r,j}}{a_{r,j}}.
\end{equation}
The last expression works for $a_{r,j}=0$ as well.

Indeed $(a_{r,j}-1)!$ is the total number of directed cycles formed by $a_{r,j}$ elements from $[\ell]$; $b_{r,j}!$ is the total number of ways to order, linearly, $b_{r,j}$ elements from
$[N]\setminus \ell$, and $\binom{b_{r,j}-1}{a_{r,j}-1}$ is the total number of ways to break any
such $b_{r,j}$-long sequence into $a_{r,j}$ blocks of positive lengths to be fitted between 
$a_{r,j}$ cyclically arranged elements from $[\ell]$, starting with the smallest element among them
and moving in the cycle's direction, say. 

Therefore
\begin{equation}\label{Q4nu}
\begin{aligned}
Q(\vec{\nu},\ell)&=(N-\ell)!\,\ell!\sum_{\bold a,\,\bold b\text{ meet }\atop
\eqref{restr1}, \eqref{restr2},\eqref{sums}}\prod_{r\ge1}\frac{1}{\nu_r!}\prod_{j\le \nu_r}
\frac{c(a_{r,j},b_{r,j})}{a_{r,j}!\,b_{r,j}!}\\
&=(N-\ell)!\,\ell!\,[w^{\ell}]\prod_{r\ge 1}\frac{1}{\nu_r!}\,\left(\sum_{b>0,\, a+b=r}
\frac{1}{b}\binom{b}{a}w^a\right)^{\nu_r}.
\end{aligned}
\end{equation}
Having found $Q(\vec{\nu},\ell)$, we turn to $p(N,\ell ,k)$, the probability that $\sigma$ blocks the elements of $[\ell]$.
Using \eqref{chila*(nu)}, the equality $\nu=\sum_r \nu_r$, and  and \eqref{Q4nu}, we obtain 
\begin{multline*}
\sum_{s:\, \vec{\nu}(s)=\vec{\nu}}\chi^{\la^*}(s)=(-1)^{\la^1}(N-\ell)!\,\ell!\\
\times 
[\xi^{\la_1}w^{\ell}]\,\,\frac{\xi}{1-\xi}\,\prod_r\frac{1}{\nu_r!}\left[-(1-\xi^r)\left(\sum_{b>0,
\atop a+b=r} 
\frac{1}{b}\binom{b}{a}w^a\right)\right]^{\nu_r}.
\end{multline*}
Call a permutation $s$ of $[N]$ admissible if it blocks the elements of $[\ell]$. The above identity implies
 \begin{multline}\label{sumcharadmsss4}
\sum_{s\text{ admissible}}\chi^{\la^*}(s)=(-1)^{\la^1}(N-\ell)!\,\ell!\\
\times 
[\xi^{\la_1}w^{\ell}]\,\,\frac{\xi}{1-\xi}\, \sum_{\vec{\nu}:\atop 1\nu_1+2\nu_2+\cdots=N}\prod_r\frac{1}{\nu_r!}\left[-(1-\xi^r)
\left(\sum_{b>0,\atop a+b=r}\frac{1}{b}\binom{b}{a}w^a\right)\right]^{\nu_r}.
\end{multline}
The expression in the second line of \eqref{sumcharadmsss4} equals
\begin{eqnarray*}
& &[\xi^{\la_1}w^{\ell}x^N] \frac{\xi}{1-\xi}\,\sum_{\vec\nu\,\ge \bold 0}\prod_r\frac{(x^r)^{\nu_r}}{\nu_r!}\left[-(1-\xi^r)\left(\sum_{b>0,\, a+b=r}\frac{1}{b}\binom{b}{a}w^a\right)\right]^{\nu_r}\notag\\
&=&[\xi^{\la_1}w^{\ell}x^N] \frac{\xi}{1-\xi}\,\prod_r\sum_{\nu_r\ge 0}\frac{1}{\nu_r!}
\left[-x^r(1-\xi^r)\left(\sum_{b>0,\, a+b=r}\frac{1}{b}\binom{b}{a}w^a\right)\right]^{\nu_r}\notag\\
&=&[\xi^{\la_1}w^{\ell}x^N] \frac{\xi}{1-\xi}\,\prod_r\exp\left[-x^r(1-\xi^r)\left(\sum_{b>0,\, a+b=r}\frac{1}{b}\binom{b}{a}w^a\right)\right]\notag\\
&=&[\xi^{\la_1}w^{\ell}x^N]\frac{\xi}{1-\xi}\,\exp\left[-\sum_{r\ge 1}[x^r-(x\xi)^r]\left(\sum_{b>0,\, a+b=r}\frac{1}{b}\binom{b}{a}w^a\right)\right].\label{[xila1wellxN]=}
\end{eqnarray*}
Since 
\begin{align*}
\sum_{r\ge 1}y^r \sum_{b>0,\,a+b=r}\frac{1}{b} \binom{b}{a}w^a=&\sum_{b>0}\frac{y^b}{b}
\sum_a\binom{b}{a}(yw)^a\\
=&\sum_{b>0}\frac{y^b}{b}(1+yw)^b=\sum_{b>0}\frac{[y(1+yw)]^b}{b}\\
=&\log\frac{1}{1-y(1+yw)},
\end{align*}
the bottom part \eqref{[xila1wellxN]=} becomes 
\begin{align*}
&[\xi^{\la_1}w^{\ell}x^N] \frac{\xi}{1-\xi}\exp\left(-\log\frac{1}{1-x(1+xw)}+\log\frac{1}{1-x\xi(1+x\xi w)}
\right)\\
=&[\xi^{\la_1}w^{\ell}x^N] \frac{\xi}{1-\xi}\frac{1-x(1+xw)}{1-x\xi(1+x\xi w)}\\
=&[\xi^{\la_1}x^N] \,\frac{\xi(1-x)}{(1-\xi)(1-x\xi)}\,[w^{\ell}]\frac{1-\tfrac{x^2}{1-x}w}{1-\tfrac{(x\xi)^2}{1-x\xi}w}\\
=&[\xi^{\la_1}x^N]\,\frac{\xi(1-x)}{(1-\xi)(1-x\xi)}\left[\left(\frac{(x\xi)^2}{1-x\xi}\right)^{\ell}
-\frac{x^2}{1-x}\left(\frac{(x\xi)^2}{1-x\xi}\right)^{\ell-1}\right]\\
=&[\xi^{\la_1}x^N]\,\frac{\xi}{1-x\xi}\left(\frac{(x\xi)^2}{1-x\xi)}\right)^{\ell-1}\!\!\!\frac{x^2}{1-x\xi}\,(x\xi-1-\xi)\\
=&-[\xi^{\la_1}x^N]\,\left(\frac{x^{2\ell}\xi^{2\ell-1}}{(1-x\xi)^{\ell}}+\frac{x^{2\ell}\xi^{2\ell}}{(1-x\xi)^{\ell+1}}\right)\\
=&-[\xi^{\la_1-2\ell+1}x^{N-2\ell}](1-x\xi)^{-\ell}-[\xi^{\la_1-2\ell}x^{N-2\ell}](1-x\xi)^{-\ell-1}\\
=&-\binom{N-\ell-1}{\ell-1}1_{\{\la_1=N-1\}}-\binom{N-\ell}{\ell}1_{\{\la_1=N\}}.
\end{align*}
So \eqref{sumcharadmsss4} simplifies, greatly, to 
\begin{multline}\label{sum4}
\sum_{s\text{ admissible}}\chi^{\la^*}(s)=(-1)^{\la^1-1}(N-\ell)!\,\ell!\\
\times 
\left[\binom{N-\ell-1}{\ell-1}1_{\{\la_1=N-1\}}+\binom{N-\ell}{\ell}1_{\{\la_1=N\}}\right].
\end{multline}
The rest is easy.  By \eqref{Psigma(s)},
\begin{equation}
p(N,\ell;k)\\
=\frac{1}{N!}\sum_{\la^*}(-1)^{k(\la^1-1)}\binom{N-1}{\la_1-1}^{-k+1}\sum_{s\text{ admissible}}\chi^{\la^*}(s).
\end{equation}
Combining this with \eqref{sum4} we conclude that
\begin{equation}\label{p4Nellk=}
p(N,\ell;k)=\frac{\binom{N-\ell}{\ell}}{\binom{N}{\ell}}
+(-1)^{k+1}\frac{\binom{N-\ell-1}{\ell-1}}{(N-1)^{k-1}\binom{N}{\ell}}.
\end{equation}
\end{proof}

{\bf Note.\/} The equation  \eqref{p4Nellk=} shows that $\lim_{k\rightarrow \infty} p(N,\ell;k) = {N-\ell \choose \ell} / {N \choose \ell}$,  the probability that the {\em uniformly} random permutation blocks $[\ell]$. \\

\begin{center}  {\bf Acknowledgment}  \end{center}
We are indebted to Frank Garvan, who generously helped us with matters involving Maple.
We thank an anonymous referee for suggestions on improving presentation of the results
and the additional references.

\end{document}